\documentclass[a4paper,12pt]{article}
\usepackage[utf8]{inputenc}
\usepackage[T1]{fontenc}
\usepackage{times}
\usepackage[bitstream-charter]{mathdesign}
\usepackage{epsfig}
\usepackage{enumerate}
\usepackage{color}
\RequirePackage{amsopn} % for DeclareMathOperator.
\RequirePackage{amsbsy} % for Boldsymbol
\RequirePackage{amsmath} % for multiple lines in equations
\RequirePackage{relsize} % for \mathlarger
\RequirePackage{latexsym} % for \leadsto
\RequirePackage{theorem}%
\RequirePackage{thc}%
\newcounter{mysubsection}[section]%
\newcounter{mysubsubsection}[mysubsection]%
\theorembodyfont{\slshape}%
\newtheorem{corollary}[mysubsection]{Corollary}%
\newtheorem{lemma}[mysubsection]{Lemma}%
\newtheorem{proposition}[mysubsection]{Proposition}%
\newtheorem{theorem}[mysubsection]{Theorem}%
\theorembodyfont{\upshape}%
\newtheorem{definition}[mysubsection]{Definition}%
\newtheorem{example}[mysubsection]{Example}%
\newtheorem{remark}[mysubsection]{Remark}%
\def\qed{{\unskip\nobreak\hfil\penalty50%
  \hskip2em\hbox{}\nobreak\hfil$\square$%
  \parfillskip=0pt\finalhyphendemerits=0\par}}
\newenvironment{proof}%
  {\par\addvspace{\medskipamount}%
    \upshape%
    {\slshape\scshape%\bfseries%
    Proof\hskip\labelsep}}%
  {\qed%
    \addvspace{\medskipamount}}%
\newcommand\dlim{\mathop{\oalign{\hfil$\textrm{lim}$\hfil\cr$\longrightarrow$\cr}}}
\newcommand\dom{\textrm{dom}}
\newcommand\Gr{\mathcal{G}\textit{r}}
\newcommand\Hom{\textrm{Hom}}
\newcommand\Ima{\textrm{Im}}
\newcommand\id{\textrm{id}}
\newcommand\Inf{\textrm{Inf}}
\newcommand\Ker{\textrm{Ker}}
\newcommand\rMod{\textrm{Mod}-}
\newcommand\Max{\textrm{Max}}
\newcommand\Min{\textrm{Min}}
\newcommand\sepa{\vspace*{-3mm}\setlength{\itemsep}{-2mm}}
\newcommand\Set{\mathcal{S}\textit{et}}
\newcommand\Sup{\textrm{Sup}}
\newcommand\blfootnote[1]{%
  \begingroup
  \renewcommand\thefootnote{}\footnote{#1}%
  \addtocounter{footnote}{-1}%
  \endgroup}
\usepackage[all,2cell]{xypic}\UseAllTwocells
\begin{document}

\title{Gradual and fuzzy subsets\blfootnote{\today
\newline
J. M. Garc\'{\i}a (\emph{jgarciah@ugr.es})
    Department of Applied Mathematics. University of Granada. E--18071 Granada. Spain.
\newline
P. Jara (Corresponding author) (\emph{pjara@ugr.es})
    Department of Algebra and IEMath--GR (Instituto de Matem\'{a}ticas). University of Granada. E--18071 Granada. Spain.
\newline
2020 \emph{Mathematics Subject Classification:} 03E72, 08A72, 16Y80, 20N25
\newline
\emph{Key words:} fuzzy set, gradual element, gradual set, gradual group, functorial category.}}
\author{Josefa M. Garc\'{\i}a\\Pascual Jara}
\date{}%\date{\today}
\maketitle

\begin{abstract}
In the fuzzy theory of sets and groups, the use of $\alpha$--levels is a standard to translate problems from the fuzzy to the crisp framework. Using strong $\alpha$--levels, it is possible to establish a one to one correspondence which makes possible doubly, a gradual and a functorial treatment of the fuzzy theory. The main result of this paper is to identify the class of fuzzy sets, respectively fuzzy groups, with subcategories of the functorial categories $\mathcal{S}\textit{et}^{(0,1]}$, resp. $\mathcal{G}\textit{r}^{(0,1]}$. In this line, the algebraic potential of this theory will be reached, in forthcoming papers, in the study of fuzzy modules, since, in that case, the functorial category is a well founded subcategory of a Grothendieck abelian category.
\end{abstract}

\section*{Introduction}

Let $X$ be a set, every subset $S\subseteq{X}$ is defined by its characteristic function $\chi_S:X\longrightarrow\{0,1\}$, which is defined by
$\chi_S(x)=\left\{\begin{array}{ll}
1&\textrm{ if }x\in{S},\\
0&\textrm{ if }x\notin{S}.
\end{array}\right.$
Thus, the concept of membership is exclusive. However, we can consider, in a wider environment, different degrees of membership:
1 means that the element belongs to the subset;
0 that it does not belong and
any other real number $0<\alpha<1$, would mean a different degree of membership.

The theory, in these terms, is due to Zadeh, see \cite{Zadeh:1965}, who introduces a fuzzy subset of a given set $X$ as a map $\mu:X\longrightarrow[0,1]$. From this primitive concept we can develop a whole theory of sets, relations, maps, numbers, etc.

In this approach to the fuzzy theory, we are begin by relating various mathematical theories; this relationship is evident in the crisp framework, but which in the fuzzy theory presents, so far, some difficulties.

Our approach to the fuzzy concept starts from the definition of a fuzzy element: we adopt the definition given by Duboi and Prade of a gradual element, see \cite{Dubois/Prade:2008}. So that a gradual element of a set $X$ is given by a collection of elements, each with a degree of membership, ranging in $(0,1]$: there is always an element of the set $X$ that has a degree of membership 1, and, possibly, other elements with other membership values, but never 0, that is, we do not determine any element of $X$ that has zero degree of membership.
This notion of gradual element has been extended to study several problems, see \cite{Sanchez/Delgado/Vila/Chamorro:2012}, \cite{Martin:2015} and \cite{Martin/Azvine:2013}. Nevertheless, we have preferred to maintain the former one as when applying to sets, groups, and other structures, it defines canonically a ground set, group, etc, which is an ambient object suitable for working.

For a greater flexibility in the definition we assume that not all possible degrees are reached, so a gradual element is given by a partial mapping from $(0,1]$ to $X$: we will call it a partial gradual element. If $\varepsilon$ is a partial gradual element with definition domain $\dom(\varepsilon)\subseteq(0,1]$, for its study we need to relate partial gradual elements to each other. The problem that arises is: when two gradual elements $\varepsilon_1$ and $\varepsilon_2$ are equal? It is clear that we can only compare $\varepsilon_1$ and $\varepsilon_2$ where they are defined, that is, in $\dom(\varepsilon_1)\cap\dom(\varepsilon_2)$.

This definition of equality is too weak. In fact we are more interested in knowing if $\varepsilon_1$ and $\varepsilon_2 $ take equal values in a range $[\alpha,1]$, for some $\alpha\in(0,1]$. Taking into account that, whenever $\beta\in(0,1]$ is very small, it  is not relevant at all if that $\varepsilon_1(\beta)$ and $\varepsilon_2(\beta)$ are the same or different; we are more interested in knowing whenever $\varepsilon_1$ and $\varepsilon_2$ coincide for values of $\beta$ close to 1.

Thus, we extend the equality relation to the case, previously indicated, of values in an interval $[\alpha,1]$. In this way a relationship is obtained: $\varepsilon_1{R_\alpha}\varepsilon_2$ if
$$
{\varepsilon_1}_{|[\alpha,1]\cap\dom(\varepsilon_1)\cap\dom(\varepsilon_2)}
={\varepsilon_2}_{|[\alpha,1]\cap\dom(\varepsilon_1)\cap\dom(\varepsilon_2)}.
$$
But this relation is not necessarily an equivalence relation, because it depends heavily on $\dom(\varepsilon_1)\cap\dom(\varepsilon_2)$. So, if we want to compare partial gradual elements, we must standardize the definition domain.  In other words, we must, for instance, extend $\dom(\varepsilon)$ to the whole $(0,1]$.

There is a standard method of doing this, consisting of, given $\alpha\in\dom(\varepsilon)$ such that $\varepsilon$ is not defined in $(\beta,\alpha)$, defining $\varepsilon(\gamma)=\varepsilon(\alpha)$ for all $\gamma\in(\beta,\alpha)$. The condition, that has seemed most efficient forces to restrict the partial gradual elements to those whose definition domain verifies that for every $\alpha\in(0,1]$ there is a minimum $\zeta$ of $[\alpha,1]\cap\dom(\varepsilon)$, to, in this way, extend $\varepsilon$ to all $(0,1]$, defining $\varepsilon(\alpha)=\varepsilon(\zeta)$. We have called inf--compact the subsets of $(0,1]$ {containing 1,} and verifying this property. In this way, every partial gradual element $\varepsilon$, with inf--compact definition domain, can be extended, in a unique way, to a gradual element $\overline{\varepsilon}$ with definition domain $(0,1]$. We call $\overline{\varepsilon}$ the extended gradual element of $\varepsilon$.

We define a total gradual element as {a map} $\varepsilon:(0,1]\longrightarrow{X}$, among which we have the extended of the partial gradual elements; and denote by $\mathcal{X}$ the set of all total gradual elements of $X$. Observe that, when working with total gradual elements the relation $R_\alpha$ is an equivalence relation.

The next step of complexity is to consider a binary operation $*$ in the set $X$, and extend it to gradual elements. The standard method is to define $(\varepsilon_1*\varepsilon_2)(\alpha)=\varepsilon_1(\alpha)*\varepsilon_2(\alpha)$ for any $\alpha\in(0,1]$.

We have that if $(X,*)$ has a more complex algebraic structure, for example, if it is a group{, a semilattice, etc,} the set $\mathcal{X}$ of all gradual elements can have the same property. However, this has not been the line we followed for the study of fuzzy structures in a set $X$, the reason is that when considering, for example, a ring structure in $X$, although $\mathcal{X}$ has a ring structure, this is of little interest, since it has too many zero--divisor elements.

We have chosen, therefore to consider a greater degree of abstraction, and consider, given a group $(G,*)$, {not the set of elements of $G$, but }the set $\mathcal{S}(G)$ of all the subgroups of $G$. We have an inclusion $\mathcal{S}(G)\subseteq\mathcal{P}(G)$, in the powerset of $G$, and the elements of $\mathcal{S}(G)$ are the non-empty subsets $S\in\mathcal{P}(G)$ verifying: $S*S\subseteq{S}$ and $S^{-1}\subseteq{S}$. When considering gradual elements $\sigma,\sigma_1,\sigma_2$ of $\mathcal{P}(G)\setminus\{\varnothing\}$, we have new gradual elements: $\sigma_1*\sigma_2$ and $\sigma^{-1}$, and naturally the notions of gradual subgroup and gradual subset appear.

A gradual subset of a set $X$ is a gradual element $\sigma$ of $\mathcal{P}(X)$, and a gradual subgroup of a group $G$ is a gradual element $\sigma$ of $\mathcal{P}(X)\setminus\{\varnothing\}$ which is a subgroup, i.e, it will be a gradual element of $\mathcal{S}(G)$. Observe that in these situations we have solved the problem of extending partial gradual subsets or subgroups, because we can define the image of any element in $(0,1]\setminus\dom(\sigma)$ equals either $\varnothing$, for subsets, or $\{e\}$, for subgroups. Therefore, in section 2 and 3 we shall consider only total gradual subsets and subgroups.

This study will lead us to relate {gradual subgroups with fuzzy subgroups, gradual groups with fuzzy groups}, and the same process will allow to relate other structures: rings, modules, etc.

Before carrying out this work we have considered necessary to implement an in-depth study that relates gradual and fuzzy {sets and} subsets.

In the set $\mathcal{X}$ of the gradual subsets of $X$ we define a closure operator $\sigma\mapsto\sigma^c=\cup\{\sigma(\beta)\mid\;\beta\geq\alpha\}$. A gradual subset $\sigma$ will be a decreasing gradual subset if $\sigma=\sigma^c$. {And in the set $\mathcal{J}(X)$ of all decreasing gradual subsets of $X$} we define an interior operator  $\sigma\mapsto\sigma^d=\cup\{\sigma(\beta)\mid\;\beta>\alpha\}$, a decreasing gradual subset $\sigma$ will be a strict decreasing gradual subset whenever $\sigma=\sigma^d$.

Associated to any fuzzy subset $\mu$ of $X$ we have a decreasing gradual subset $\sigma(\mu)$, defined $\sigma(\mu)(\alpha)=\mu_\alpha$, the $\alpha$--level of $\mu$, for any $\alpha\in(0,1]$, and a strict decreasing gradual subset $\widetilde{\sigma}(\mu)=\sigma(\mu)^d$, which is the strong $\alpha$--level, or strong $\alpha$-cut, of $\mu$. The map $\mu\mapsto\sigma(\mu)$ does not preserve unions of infinite families, and the map $\mu\mapsto\widetilde{\sigma}(\mu)$ does not preserve infinite intersections; hence after modifying the intersection, we establish an injective correspondence, {preserving union and intersection}, from fuzzy subsets to strict decreasing gradual subsets, and find conditions on strict decreasing gradual subsets to be in the image of this map; that condition is property (inf--F). Which is important, in this situation, is that we have an isomorphism, {for intersections and unions}, between fuzzy subsets and strict decreasing gradual subsets satisfying property (inf--F). As a consequence properties on fuzzy subsets can be studied via strict decreasing gradual subsets.

In addition, we consider a generalization of the theory of gradual subsets through the use of {contravariant} functors from the category $(0,1]$ to the category $\Set$ of sets which allow a functorial framework of both theories of gradual and fuzzy subsets.

This theory, {first} developed in a context of sets, can be carried out to the more algebraic framework of groups, in which we may establish also a bijection between {fuzzy subgroups} and a specific class of gradual subgroups and {contravariant} functors. In particular, this bijection will allow a functorial treatment of fuzzy groups.

\medskip

This paper is organized in three sections. In the first one we study and establish the general theory of gradual elements and introduce binary operations in the set of all gradual elements defined from binary operation in the ground set $X$. In particular, if we start from a group $G$, we get a structure of group {in the set} of gradual elements. Not in all cases this structure reflects the properties of $G$ and its elements.

For this reason, to make an algebraic development later, in the second section we study gradual subsets and operators in the set of gradual subsets that will allow to establish a close relationship, an isomorphism, between the set of fuzzy subsets and a set of gradual subsets. This study ends in Theorem~\eqref{th:20180417} in which an isomorphism is established; observe that to obtain the isomorphism we had to make use of the strict decreasing gradual subsets. To do that, first we consider binary operations in $\mathcal{P}(X)$, the power set of $X$: the standard ones are the meet (intersection) and the join (union), and translate them to gradual subsets, which are noting more than gradual elements of $\mathcal{P}(X)$.
In this section we also identify a new type of objects through the use of {contravariant} functors from the category $(0,1]$ to the category of sets. These contravariant functors, which are identified with directed systems, generalize gradual subsets and fuzzy subsets, and allow a functorial framework of these two examples, {which} will provide a tool capable of dealing with other types of gradual and fuzzy objects such as groups, rings, etc., and that will allow to work, by using direct limits, with gradual and fuzzy sets, instead of with gradual and fuzzy subsets.

The third section is devoted to study the more complex example of gradual groups. After studying the different concepts related to group theory, we establish the most important result, Theorem~\eqref{th:201804b}, showing a bijection between equivalence classes of fuzzy subgroups and some specific strict gradual subgroups. This gradual subgroups appear in a natural way after studying two operator on gradual subgroups, one a closure operator and another one an interior operator in the class of all decreasing gradual groups. The formulation of the theory in terms of operators allows to develop a more abstract framework, in this case a functorial one, and hence, obtain new properties and relationships between known objects.

\nocite{Fortin/Dubois/Fargier:2008}
\nocite{Lewis/Martin:2013}
\nocite{Martin/Azvine:2013}
\nocite{Martin:2015}
\nocite{MORA-CAMINO/NUNEZ:2018}
\nocite{Nadaban:2016}
\nocite{Piegat/Plucinski:2015}
\nocite{Sanchez/Delgado/Vila/Chamorro:2012}
\nocite{Wang/Li:2017}

\section{Gradual elements}

\subsection{Gradual elements}

\begin{definition}
Let $X$ be {a} set, a \textbf{{total} gradual element} of $X$, is a map $\varepsilon:(0,1]\longrightarrow{X}$, and a \textbf{partial gradual element} of $X$ is a map $\varepsilon:L\subseteq(0,1]\longrightarrow{C}$, defined on a subset $L\subseteq(0,1]$ such that $1\in{L}$.
For simplicity, depending of the context, we use \textbf{gradual element} to refer either to a total gradual element or to a partial gradual element.
\end{definition}

For any partial gradual element $\varepsilon:L\subseteq(0,1]\longrightarrow{X}$ we call $L$ the \textbf{definition domain} of $\varepsilon$, and represent it by $\dom(\varepsilon)$.
We represent by $\mathcal{X}$ the set of all total gradual elements of $X$, and by $\underline{\mathcal{X}}$ the set of all partial gradual elements.

A gradual element $\varepsilon'$ is an \textbf{extension} of the gradual element $\varepsilon$ if $\varepsilon'_{|\dom(\varepsilon)}=\varepsilon$.

There is a particularly useful method of extending a partial gradual element $\varepsilon$ to a {total} gradual one, this is the case in which for any $\alpha\in(0,1]$ there exists $\Min([\alpha,1]\cap\dom(\varepsilon))$; then we define a new gradual element $\overline{\varepsilon}$ as follows:
\[
\overline{\varepsilon}(\alpha)=\varepsilon(\zeta),\textrm{ being }\zeta=\Min([\alpha,1]\cap\dom(\varepsilon)).
\]

See Example~\eqref{ex:20181210} in which examples of extensions of partial gradual elements appear. Another example is provided whenever we consider the partial gradual element $\varepsilon:\{1\}\longrightarrow\{a,b\}$, defined by $\varepsilon(1)=a$. In this case an extension $\overline{\varepsilon}:(0,1]\longrightarrow\{a,b\}$ is defined by $\overline{\varepsilon}(\alpha)=a$ for any $\alpha\in(0,1]$, the constant map equals to $a$.

A subset $L\subseteq(0,1]$, {containing 1}, such that there exists $\Min([\alpha,1]\cap{L})$, for any $\alpha\in(0,1]$, is named an \textbf{inf--compact} subset of $(0,1]$.

The following are examples of inf--compact subsets of $(0,1]$:
\begin{enumerate}[(1)]\sepa
\item
Any compact subset $C\subseteq(0,1]$, containing 1, is inf--compact. In particular, any finite subset and any closed subset of $(0,1]$, containing 1, are inf--compact.
\item
Any ascending sequence in $(0,1]$, union with $\{1\}$, is inf--compact.
\item
Any interval $[a,b)\subseteq(0,1]$, union with $\{1\}$, is inf--compact.
\item
Any union of finitely many inf--compact subsets is inf--compact.
\end{enumerate}

In the following, the domain of any partial gradual element will be an inf--compact subset of $(0,1]$, containing $1$; whence, any partial gradual element can be extended to a total gradual element.

{\begin{lemma}
If $\{C_i\mid\;i\in{I}\}$ is a family of inf--compact subsets, containing $1$, then $\cap_iC_i$ is inf--compact.
\end{lemma}
\begin{proof}
For any $\alpha\in(0,1]$ let $\xi=\Inf([\alpha,1]\cap(\cap_iC_i))$, and define $\xi_i=\Min([\xi,1]\cap{C_i})$, hence $\xi\leq\xi_i$, for any $i\in{I}$. On the other hand, since $[\alpha,1]\cap(\cap_iC_i)=[\xi,1]\cap(\cap_iC_i)\subseteq[\xi,1]\cap{C_i}$, then $\Min([\xi,1]\cap{C_i})=\Inf([\xi,1]\cap{C_i})\leq\Inf([\xi,1]\cap(\cap_iC_i))=\xi$. In consequence, $\xi=\xi_i$ for any $i\in{I}$.
\end{proof}}

For any element $a\in{X}$ there exists a partial gradual element, which we represent by $\varepsilon_a$, with $\dom(\varepsilon_a)=\{1\}$, and defined by $\varepsilon_a(1)=a$. {We denote also by $\varepsilon_a$ the extension $\overline{\varepsilon_a}$}. Without lost of generality we may identify the element $a\in{X}$ and the gradual element $\varepsilon_a\in\mathcal{X}$, and denote them simply by $a$.

In this way, a gradual element is nothing more than a collection of elements of $X$, each one with a degree of membership; thus if $\varepsilon$ is a gradual element then $\varepsilon(\alpha)$ is an element of $X$ with membership degree $\alpha$. Since $\varepsilon(1)$ is always defined, we have it is an element of $X$ with the highest membership degree; since $\varepsilon(0)$ is not defined, then there is not any element with zero membership degree.

\subsection{Relations between gradual elements}

For any $\alpha\in[0,1]$, in the set of all partial gradual elements we define a relation $R_\alpha$ as: for partial gradual elements $\varepsilon_1$ and $\varepsilon_2$ of $X$ we say $\varepsilon_1R_\alpha\varepsilon_2$ if
$$
{\varepsilon_1}_{|[\alpha,1]\cap\dom(\varepsilon_1)\cap\dom(\varepsilon_2)}
={\varepsilon_2}_{|[\alpha,1]\cap\dom(\varepsilon_1)\cap\dom(\varepsilon_2)}.
$$
Observe that if $\alpha\in(0,1]$, and $\varepsilon_1,\varepsilon_2$ are total gradual elements then $\varepsilon_1R_\alpha\varepsilon_2$, whenever ${\varepsilon_1}_{|[\alpha,1]}={\varepsilon_2}_{|[\alpha,1]}$.

\begin{lemma}
For any $\alpha,\beta\in[0,1]$ we have:
\begin{enumerate}[(1)]\sepa
\item
If $\alpha,\beta\in[0,1]$ satisfy $\alpha\leq\beta$, then $R_\alpha\subseteq{R_\beta}$.
\item
The relation $R_\alpha$ is an equivalence relation in the set of all total gradual elements.
\end{enumerate}
\end{lemma}

The equivalence relation $R_\alpha$ indicates us when two gradual elements are equal at a certain level. For instance,
\begin{enumerate}[(1)]\sepa
\item
if $\alpha=1$, then we only have an equivalence class for each element of $X$;
\item
if $\alpha=0$, then two gradual elements belong to the same equivalence class if, and only if, they coincide in their definition domains.
\end{enumerate}

It is necessary to remark that these equivalence relations $R_\alpha$ are not compatible with the extension process. Indeed, if $\varepsilon_1{R_\alpha}\varepsilon_2$, not necessarily $\overline{\varepsilon}_1R_\alpha\overline{\varepsilon}_2$ as the following example shows.

{\begin{example}
Let $X=\{a,b\}$, and define
$$
\varepsilon_1(\delta)=\left\{\begin{array}{ll}
b,&\textrm{ if }\delta=1,\\
a,&\textrm{ if }\frac{1}{2}\leq\delta<1,
\end{array}\right.
\qquad
\varepsilon_2(\delta)=\left\{\begin{array}{ll}
b,&\textrm{ if }\delta=1,\\
a,&\textrm{ if }0<\delta\leq\frac{1}{2}.
\end{array}\right.
$$
Then $\varepsilon_1R_\alpha\varepsilon_2$, for any $\alpha\in(0,1]$, but $\overline{\varepsilon}_1R_\alpha\overline{\varepsilon}_2$ if, and only if, $\alpha=1$.
\end{example}}

Let $f:X\longrightarrow{Y}$ be a map between two sets;
\begin{enumerate}[(1)]\sepa
\item
for any total gradual (resp. partial gradual) element $\varepsilon$ of $X$ we have a total gradual (resp. partial gradual) element of $Y$ defined by the composition $f\varepsilon:\dom(\varepsilon)\longrightarrow{X}\longrightarrow{Y}$; we call $f\varepsilon$ the \textbf{image} of $\varepsilon$ by $f$;
\item
for any gradual element $\varepsilon'$ of $Y$ a gradual element $\varepsilon$ of $X$ is an \textbf{inverse image} of $\varepsilon'$ if $\varepsilon'=f\varepsilon$.
\end{enumerate}

\subsection{Binary operations and gradual elements}

There is another method to relate gradual elements of a set $X$, this is the case in which there exists a binary operation in $X$.

Let $G$ be a set together a binary operation, say $*$, and $\varepsilon_1,\varepsilon_2$ gradual elements of $G$, we define a new gradual element $\varepsilon_1*\varepsilon_2$ as:
$$
(\varepsilon_1*\varepsilon_2)(\alpha)=\varepsilon_1(\alpha)*\varepsilon_2(\alpha),
\textrm{ for any }\alpha\in\dom(\varepsilon_1)\cap\dom(\varepsilon_2).
$$
In the case of partial gradual elements $\varepsilon_1,\varepsilon_2$ we have that $\dom(\varepsilon_1*\varepsilon_2)=\dom(\varepsilon_1)\cap\dom(\varepsilon_2)$.

This operation non necessarily is compatible with the extension construction.

The following example shows that if we start from two partial gradual elements $\varepsilon_1$ and $\varepsilon_2$, then not necessarily we have the equality: $\overline{\varepsilon_1*\varepsilon_2}=\overline{\varepsilon}_1*\overline{\varepsilon}_2$, i.e., then extension map is not necessarily a homomorphism with respect to the binary operation $*$.

\begin{example}\label{ex:20181210}
Let $\varepsilon_1,\varepsilon_2$ be partial gradual elements defined on $\mathbb{Z}$, defined as:
\[
\begin{array}{c}
\varepsilon_1(\alpha)=\left\{\begin{array}{ll}
2,&\textrm{ if }\frac{1}{2}\leq\alpha\leq1,\\
1,&\textrm{ if }\frac{1}{10}\leq\alpha\leq\frac{1}{3},\\
\end{array}\right.
\qquad
\varepsilon_2(\alpha)=2,\textrm{ if }\frac{2}{3}\leq\alpha\leq1.
\end{array}
\]
In this case we have
\[
\begin{array}{l}
(\varepsilon_1+\varepsilon_2)(\alpha)=4\textrm{, if }\frac{2}{3}\leq\alpha\leq1,
\end{array}
\]
and the extended gradual elements are
\[
\overline{\varepsilon_1}(\alpha)=\left\{\begin{array}{ll}
2,&\textrm{ if }\frac{1}{3}<\alpha\leq1,\\
1,&\textrm{ if }\alpha\leq\frac{1}{3},\\
\end{array}\right.
\quad
\overline{\varepsilon_2}(\alpha)=2,\textrm{ if }\alpha\leq1,
\quad
\overline{(\varepsilon_1+\varepsilon_2)}(\alpha)=4,\textrm{ if }\alpha\leq1.
\]
On the other hand, we have
\[
(\overline{\varepsilon_1}+\overline{\varepsilon_2})(\alpha)=\left\{\begin{array}{ll}
4,&\textrm{ if }\frac{1}{3}<\alpha\leq1,\\
3,&\textrm{ if }\alpha\leq\frac{1}{3}.\\
\end{array}\right.
\]
{\[
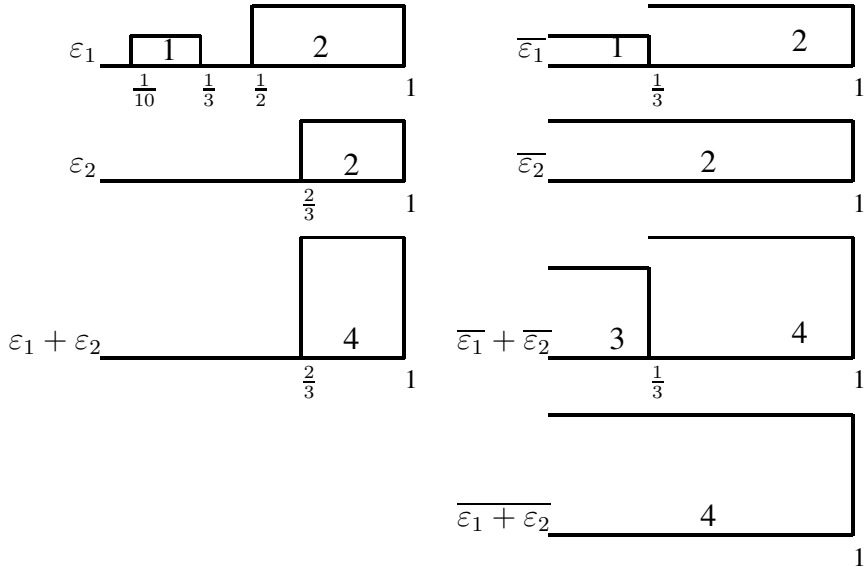
\begin{array}{lll}
\setlength{\unitlength}{.4cm}
\begin{picture}(10,3)(0,0)
\linethickness{1pt}
\put(-1,.3){\large{$\varepsilon_1$}}
\put(0,0){\line(1,0){10}}
\put(10,0){\line(0,1){2}}
\put(5,0){\line(0,1){2}}
\put(5,2){\line(1,0){5}}
\put(7,.3){\large{2}}
\put(3.3,0){\line(0,1){1}}
\put(1,0){\line(0,1){1}}
\put(1,1){\line(1,0){2.3}}
\put(2,1.2){\large{1}}
\put(10,-1){1}
\put(5,-1){$\frac{1}{2}$}
\put(3.3,-1){$\frac{1}{3}$}
\put(1,-1){$\frac{1}{10}$}
\end{picture}    &\qquad\qquad\quad  &\setlength{\unitlength}{.4cm}
\begin{picture}(10,3)(0,0)
\linethickness{1pt}
\put(-1.5,.3){\large{$\overline{\varepsilon_1}$}}
\put(0,0){\line(1,0){10}}
\put(10,0){\line(0,1){2}}
\put(3.3,0){\line(0,1){1}}
\put(3.3,2){\line(1,0){6.7}}
\put(8,.5){\large{2}}
\put(3.3,0){\line(0,1){1}}
\put(0,1){\line(1,0){3.3}}
\put(2,1.2){\large{1}}
\put(10,-1){1}
\put(3.3,-1){$\frac{1}{3}$}
\end{picture} \\
\setlength{\unitlength}{.4cm}
\begin{picture}(10,3.5)(0,0)
\linethickness{1pt}
\put(-1,.3){\large{$\varepsilon_2$}}
\put(0,0){\line(1,0){10}}
\put(10,0){\line(0,1){2}}
\put(6.6,0){\line(0,1){2}}
\put(6.6,2){\line(1,0){3.4}}
\put(8,.2){\large{2}}
\put(10,-1){1}
\put(6.6,-1){$\frac{2}{3}$}
\end{picture}
   &\qquad\qquad\quad  &\setlength{\unitlength}{.4cm}
\begin{picture}(10,3.5)(0,0)
\linethickness{1pt}
\put(-1,.3){\large{$\overline{\varepsilon_2}$}}
\put(0,0){\line(1,0){10}}
\put(10,0){\line(0,1){2}}
\put(0,2){\line(1,0){10}}
\put(5,.3){\large{2}}
\put(10,-1){1}
\end{picture}
\\\\
\setlength{\unitlength}{.4cm}
\begin{picture}(10,4.5)(0,0)
\linethickness{1pt}
\put(-3,.3){\large{$\varepsilon_1+\varepsilon_2$}}
\put(0,0){\line(1,0){10}}
\put(10,0){\line(0,1){4}}
\put(6.6,0){\line(0,1){4}}
\put(6.6,4){\line(1,0){3.4}}
\put(8,.3){\large{4}}
\put(10,-1){1}
\put(6.6,-1){$\frac{2}{3}$}
\end{picture}
 &\qquad\qquad\quad  &\setlength{\unitlength}{.4cm}
\begin{picture}(10,4.5)(0,0)
\linethickness{1pt}
\put(-3,.3){\large{$\overline{\varepsilon_1}+\overline{\varepsilon_2}$}}
\put(0,0){\line(1,0){10}}
\put(10,0){\line(0,1){4}}
\put(3.3,0){\line(0,1){1}}
\put(3.3,4){\line(1,0){6.7}}
\put(8,.5){\large{4}}
\put(3.3,0){\line(0,1){3}}
\put(0,3){\line(1,0){3.3}}
\put(2,.3){\large{3}}
\put(10,-1){1}
\put(3.3,-1){$\frac{1}{3}$}
\end{picture}
\\\\
             &                   &\setlength{\unitlength}{.4cm}
\begin{picture}(10,4.5)(0,0)
\linethickness{1pt}
\put(-3,.3){\large{$\overline{\varepsilon_1+\varepsilon_2}$}}
\put(0,0){\line(1,0){10}}
\put(10,0){\line(0,1){4}}
\put(0,4){\line(1,0){10}}
\put(5,.3){\large{4}}
\put(10,-1){1}
\end{picture}

\end{array}
\]}
\end{example}

On the other hand, this operation is compatible with the equivalence relations $R_\alpha$.

\begin{lemma}\label{le:20180420}
Let $G$ be a set together a binary operation $*$, for any $\alpha\in[0,1]$ the relations $R_\alpha$ { in the set of all total gradual elements (resp. in the set of all partial gradual elements)} are \textbf{compatible} with the binary operation, i.e., for gradual elements $\varepsilon,\varepsilon_1,\varepsilon_2$ of $G$, if $\varepsilon_1R_\alpha\varepsilon_2$ then $(\varepsilon*\varepsilon_1)R_\alpha(\varepsilon*\varepsilon_2)$ and $(\varepsilon_1*\varepsilon)R_\alpha(\varepsilon_2*\varepsilon)$.
\end{lemma}
\begin{proof}
Let $\varepsilon_1,\varepsilon_2,\varepsilon$ be gradual elements such that $\varepsilon_1{R_\alpha}\varepsilon_2$, for any $\beta\in[\alpha,1]\cap\dom(\varepsilon_1)\cap\dom(\varepsilon_2)\cap\dom(\varepsilon)$ we have:
\[
(\varepsilon_1*\varepsilon)(\beta)
=\varepsilon_1(\beta)*\varepsilon(\beta)
=\varepsilon_2(\beta)*\varepsilon(\beta)
=(\varepsilon_2*\varepsilon)(\beta).
\]
\end{proof}

In some cases, in which $G$ {has a richer structure}, this structure could be inherited by the sets of  gradual elements. Let us show an example.

\begin{lemma}\label{le:20181108}
Let $G$ be a group, with binary operation $*$ and neutral element $e$, the following statements hold:
\begin{enumerate}[(1)]\sepa
\item
{The set $\mathcal{G}$ of all total gradual elements} is a group with neutral element $e$, i.e., the total gradual element $\varepsilon_e$.
\item
For any $\alpha\in[0,1]$ we have that $\mathcal{G}/R_\alpha$ is a group.
\end{enumerate}
\end{lemma}
\begin{proof}
For any $\varepsilon,\varepsilon_1,\varepsilon_2\in\mathcal{G}$ we define, for any $\alpha\in(0,1]$:
\[
(\varepsilon_1*\varepsilon_2)(\alpha)=\varepsilon_1(\alpha)*\varepsilon_2(\alpha),
\quad\textrm {and }
\quad
(\varepsilon^{-1})(\alpha)=\varepsilon(\alpha)^{-1}.
\]
\par
(1). %
In $\mathcal{G}$ the operation is associative and $e$ is the neutral element. For any $\varepsilon\in\mathcal{G}$ we have $\varepsilon^{-1}$ is the inverse of $\varepsilon$. Therefore $\mathcal{G}$ is a group, and it is abelian whenever $G$ is.
\par (2). %
It is a direct consequence of being $R_\alpha$ a compatible equivalence relation.
\end{proof}

\begin{remark}[The particular case of partial gradual elements]
Let $G$ be a group, in the set $\underline{\mathcal{G}}$ of all partial gradual elements we have an associative binary operation, $(\varepsilon_1*\varepsilon_2)(\alpha)=\varepsilon_1(\alpha)*\varepsilon_2(\alpha)$ for any $\alpha\in\dom(\varepsilon_1)\cap\dom(\varepsilon_2)$, but we have ``many'' possible neutral elements. Thus, to get a useful structure we must define before an equivalence relation to put together all of them. For instance, given two partial gradual elements $\varepsilon_1,\varepsilon_2$, since $\varepsilon_1*\varepsilon_2$ is defined on $\dom(\varepsilon_1)\cap\dom(\varepsilon_2)$, there are three possible neutral elements:
$e_{|\dom(\varepsilon_1)}, e_{|\dom(\varepsilon_2)}$ and finally $e_{|\dom(\varepsilon_1)\cap\dom(\varepsilon_2)}$, which are different two to two.

We can try to fix this problem in defining an equivalence relation $R$ in $\underline{\mathcal{G}}$ generated by:
$$
\varepsilon_1R\varepsilon_2\textrm{, if there is an inf--compact subset }C\subseteq(0,1]\textrm{ containing 1, such that }{\varepsilon_1}_{|C}={\varepsilon_2}_{|C}.
$$
With the relation $R$, the problem is that we may have $\dom(\varepsilon_1)\cap\dom(\varepsilon_2)=\{1\}$, and this trivialize this relation.

Hence, to obtain a well defined structure on partial gradual elements we may consider only special types of partial gradual elements, for instance, the subset of $\underline{\mathcal{G}}$ constituted by those partial gradual elements who have the same (inf--compact) domain containing 1.
\end{remark}

Thus we can extend the above Lemma~\eqref{le:20181108} to consider gradual elements defined on an inf--compact subset containing 1.

\begin{proposition}
Let $G$ be a  group, and let $C\subseteq(0,1]$ be a inf--compact subset containing 1. If $\underline{\mathcal{G}}$ be the set of all partial gradual elements whose domain is $C$, the following statements hold:
\begin{enumerate}[(1)]\sepa
\item\label{it:000}
For any $\alpha\in{C}$ the relation $R_\alpha$ is an equivalente relation in $\underline{\mathcal{G}}$.
\item\label{it:001}
In $\underline{\mathcal{G}}$ we have an associative operation.
\item\label{it:002}
The extending map from $\underline{\mathcal{G}}$ to $\mathcal{G}$ is a group monomorphism.
\item\label{it:003}
For any $\alpha\in{C}$ the equivalence relation $R_\alpha$ in $\underline{\mathcal{G}}$ is compatible.
\item\label{it:004}
The groups $\underline{\mathcal{G}}$ and $\underline{\mathcal{G}}/R_\alpha$ are abelian whenever $G$ is.
\end{enumerate}
\end{proposition}
\begin{proof}
\eqref{it:000}. %
It is reflexive and symmetric, and obviously it is transitive as the domain is the whole set $C$.
\par
\eqref{it:001}. %
It is evident as the product is defined componentwise.
\par
\eqref{it:002}. %
Let $\varepsilon_1,\varepsilon_2\in\underline{\mathcal{G}}$, and $\xi=\Min([\alpha,1]\cap{C})$, then
\[
(\overline{\varepsilon_1*\varepsilon_2)}(\alpha)
=(\varepsilon*\varepsilon_2)(\xi)
=\varepsilon_1(\xi)*\varepsilon_2(\xi)
=\overline{\varepsilon}_1(\alpha)*\overline{\varepsilon}_2(\alpha).
\]
\par
\eqref{it:003}. %
It is similar to the proof on Lemma~\eqref{le:20180420}.
\par
\eqref{it:004}. %
It is evident as the product is defined componentwise.
\end{proof}

It is clear that it is better to consider total gradual elements instead of partial gradual elements, and therefore work in $\mathcal{G}$.

If the group $\mathcal{G}$ has $e$ as neutral element, and for any $\alpha$ in $[0,1]$ we consider the equivalence relation $R_\alpha$, we may rewriting Lemma~\eqref{le:20180420} obtaining a filtration of subgroups of $\mathcal{G}$.

\begin{proposition}
Let $G$ be a group with neutral element $e$, if for any $\alpha\in[0,1]$ we define
$$
\mathcal{G}_\alpha=\{\varepsilon\in\mathcal{G}\mid\;\varepsilon{R_\alpha}e\},
$$
then we have:
\begin{enumerate}[(1)]\sepa
\item
For each $\alpha\in(0,1]$ the subgroup $\mathcal{G}_\alpha\subseteq\mathcal{G}$ is a normal.
\item
There is a filtration $\{\mathcal{G}_\alpha\mid\;\alpha\in[0,1]\}$ where $\mathcal{G}_\alpha\subseteq\mathcal{G}_\beta$ is $\alpha\leq\beta$.
\item
We have inclusions:
$\mathcal{G}_0\subseteq\mathcal{G}_\alpha\subseteq\mathcal{G}_1\subseteq\mathcal{G}$,
and surjective group homomorphisms:
$$
\mathcal{G}/\mathcal{G}_0
\longrightarrow\mathcal{G}/\mathcal{G}_\alpha
\longrightarrow\mathcal{G}/\mathcal{G}_1\cong{G}.
$$
\end{enumerate}
\end{proposition}

Observe that in all these examples it seems that the way to define an operation on gradual elements is to define it componentwise.

If the base set $X$ has a more richer structure; for instance, if it is a ring $R$, then the corresponding sets $\mathcal{R}$ and $\underline{\mathcal{R}}$ are rings, but there are in these rings many elements which are zero--divisor. So, in this case, the use of gradual elements is not a good option. For that, in this and forthcoming papers, we shall develop a different approach to study algebraic structures. Before doing that, let us study the simplest notion of gradual subset, and after doing this we shall return to consider a set endowed with one or several binary operations, for instance a group.

\section{Gradual subsets}\label{se:20181115}

Once we have established the notion of gradual element of a set $X$, we shall apply it to define new objects. If we consider a set $X$ and the power set $\mathcal{P}(X)$, we can study gradual elements of $\mathcal{P}(X)$, thereby the concept of gradual subset appears.

\subsection{Gradual subsets}

{\begin{definition}
Let $X$ be a set, and let $\mathcal{P}(X)$ be the power set of $X$, i.e., $\mathcal{P}(X)=\{S\mid\;S\textrm{ is a subset of }X\}$. We define a \textbf{gradual subset} of $X$ as a gradual element of $\mathcal{P}(X)$. We represent by $\sigma:(0,1]\longrightarrow\mathcal{P}(X)$ a gradual subset of $X$.
\end{definition}}

Throughout this section we follow the same assumptions used for gradual elements in the previous section. In this way, we have defined \textbf{partial gradual subsets} and \textbf{total gradual subsets}.

In some sense gradual subsets are a generalization of gradual elements. Thus, for any gradual element $\varepsilon$ and any gradual subset $\sigma$ we say $\varepsilon$ \textbf{belongs} to $\sigma$ if for any $\alpha\in\dom(\varepsilon)\cap\dom(\sigma)\subseteq[0,1]$ we have $\varepsilon(\alpha)\in\sigma(\alpha)$, and write $\varepsilon\in\sigma$. In the same way, given two gradual subsets $\sigma_1,\sigma_2$, we say that $\sigma_1$ is a \textbf{subset} of $\sigma_2$ if $\sigma_1(\alpha)\subseteq\sigma_2(\alpha)$ for any $\alpha\in\dom(\sigma_1)\cap\dom(\sigma_2)$, and write $\sigma_1\subseteq\sigma_2$.

In general, for any gradual subset $\sigma$, and elements $\alpha,\beta\in\dom(\sigma)$ such that $\alpha\leq\beta$, we have no information about the relationship of $\sigma(\alpha)$ and $\sigma(\beta)$. In some cases, as in the classical one of $\alpha$--levels in fuzzy set theory, there is an evident relationship, as we shall see later. To work with them, first we introduce the following definitions, that reflect toe order existing in $(0,1]$.

Let $\sigma$ be a gradual subset of $X$ we say $\sigma$ is
\begin{enumerate}[(1)]%\sepa
\item
\textbf{increasing} if for any $\alpha,\beta\in\dom(\sigma)$ such that $\alpha\leq\beta$, we have $\sigma(\alpha)\subseteq\sigma(\beta)$. For any increasing gradual subset $\sigma$ de $X$, if $\zeta=\Min(\dom(\sigma))$, then $\sigma(\zeta)\subseteq\sigma(\alpha)\subseteq\sigma(1)$ for any $\alpha\in\dom(\sigma)$.
\item
\textbf{decreasing} if for any $\alpha,\beta\in\dom(\sigma)$ such that $\alpha\leq\beta$, we have $\sigma(\alpha)\supseteq\sigma(\beta)$. For any decreasing gradual subset $\sigma$ de $X$, we have $\sigma(1)\subseteq\sigma(\alpha)$ for any $\alpha\in\dom(\sigma)$.
\end{enumerate}

Let us show some examples of decreasing gradual subsets.

\begin{example}
Let $\mu$ be a fuzzy subset of $X$, i.e., a map $\mu:X\longrightarrow[0,1]$ that we assume it is not constant equal to 0. For any $\alpha\in(0,1]$ we define the
\begin{enumerate}[(1)]%\sepa
\item
\textbf{$\alpha$--level} of $\mu$ as $\mu_\alpha=\{x\in{X}\mid\;\mu(x)\geq\alpha\}$.
\newline
In this case we have a decreasing gradual subset $\sigma(\mu)$, defined $\sigma(\mu)(\alpha)=\mu_\alpha$ for any $\alpha\in(0,1]$.
\item
\textbf{strict $\alpha$--level} (or \textbf{strong $\alpha$-level}) of $\mu$ as $\widetilde\mu_\alpha=\{x\in{X}\mid\;\mu(x)>\alpha\}$.
\newline
In this case we have a decreasing gradual subset $\widetilde{\sigma}(\mu)$, defined
$$
\widetilde{\sigma}(\mu)(\alpha)=\left\{\begin{array}{ll}
\widetilde\mu_\alpha,&\textrm{ for any }\alpha\in(0,1),\textrm{ and}\\
\mu_1,&\textrm{ if }\alpha=1.
\end{array}\right.
$$
\item
Let $\overline{\mu}$ the fuzzy subset defined by $\overline{\mu}(x)=1-\mu(x)$, for any $x\in{X}$; the $\alpha$--levels of $\overline{\mu}$ define a decreasing gradual subset $\sigma(\overline{\mu})(\alpha)=\{x\in{X}\mid\;\mu(x)\leq1-\alpha\}$.
\item
\textbf{inverse $\alpha$--level} of $\mu$ as $\mu^\alpha=\{x\in{X}\mid\;\mu(x)\leq\alpha\}$.
\newline
In this case we have a increasing gradual subset $\tau(\mu)$, defined $\tau(\mu)(\alpha)=\mu^\alpha$ for any $\alpha\in(0,1]$
\end{enumerate}
\end{example}

\subsection{{Operators on} gradual subsets}

The following are examples of constructions that can be carried out for any gradual subset, and which will be useful in their study.

Let $\sigma$ be a gradual subset of $X$, associated to $\sigma$ we define two new gradual subsets:
\begin{enumerate}[(1)]%\sepa
\item
The \textbf{accumulation} $\sigma^c$ of $\sigma$.
\[
\sigma^c(\alpha)=\cup\{\sigma(\beta)\mid\;\alpha\leq\beta\in\dom(\sigma)\},\textrm{ for any }\alpha\in\dom(\sigma).
\]
It is clear that for any gradual subset $\sigma$ the accumulation $\sigma^c$ is a \textbf{decreasing gradual subset}, and a gradual subset $\sigma$ is decreasing if, and only if, $\sigma=\sigma^c$.
\par
For any gradual subset $\sigma$ we have $\sigma\subseteq\sigma^c=\sigma^{cc}$.
\item
The \textbf{strict accumulation} $\sigma^d$ of $\sigma$.
\[
\sigma^d(\alpha)=\left\{\begin{array}{ll}
 \sigma(1),&\textrm{ if }\alpha=1.\\
 \cup\{\sigma(\beta)\mid\;\alpha<\beta\in\dom(\sigma)\},&\textrm{ if }\alpha\in\dom(\sigma)\setminus\{1\}.
 \end{array}\right.
\]
It is clear that for any gradual subset $\sigma$ the strict accumulation $\sigma^d$ is a decreasing gradual subset, and $\sigma^d\subseteq\sigma^c$. In general, $\sigma\nsubseteq\sigma^d$.
\end{enumerate}

Thus, we have an operator, $c$, on gradual subsets: $\sigma\mapsto\sigma^c$. The behaviour of $c$ is reflected in the following lemma.

\begin{lemma}
Let $X$ be a set, for any gradual subsets $\sigma_1,\sigma_2,\sigma$ of $X$ the following statements hold:
\begin{enumerate}[(1)]\sepa
\item
$\sigma\subseteq\sigma^c$.
\item
$\sigma^c=\sigma^{cc}$.
\item
If $\sigma_1\subseteq\sigma_2$, then $\sigma_1^c\subseteq\sigma_2^c$.
\item
$\sigma^c$ is the smallest decreasing gradual subset containing $\sigma$.
\end{enumerate}
\end{lemma}
\begin{proof}
(1), (2) and (3) are easy.
\par
(4). %
Let $\tau$ be a decreasing gradual subset such that $\sigma\subseteq\tau$, then $\sigma^c\subseteq\tau^c=\tau$.
\end{proof}

This means that the operator $c$ is a closure operator in the set $\mathcal{X}$ of all gradual subsets of $X$.

Remember that a \textbf{closure operator} in a poset ({partial ordered set}) $P$ is a map $c:P\longrightarrow{P}$ satisfying:
\begin{enumerate}[(1)]\sepa
\item
$p\leq{c(p)}$ for any $p\in{P}$.
\item
For any $p_1,p_2\in{P}$ such that $p_1\leq{p_2}$ we have $c(p_1)\leq{c(p_2)}$.
\item
$c(p)=cc(p)$ for any $p\in{P}$.
\end{enumerate}
The elements $p\in{P}$ such that $c(p)=p$ are named the \textbf{$c$--closed} elements. Thus, the gradual subsets which are closed for the operator $c$, are the decreasing gradual subsets. Let us denote by $\mathcal{J}(X)$ the set of all decreasing gradual subsets of $X$.

In the same way, we may consider the operator $d$, defined: $\sigma\mapsto\sigma^d$; its behaviour is reflected in the following lemma.

\begin{lemma}
Let $X$ be a set, for any gradual subsets $\sigma_1,\sigma_2,\sigma$ of $X$ the following statements hold:
\begin{enumerate}[(1)]\sepa
\item
$\sigma^d\subseteq\sigma^c$.
\item
If $\sigma_1\subseteq\sigma_2$, then $\sigma_1^d\subseteq\sigma_2^d$.
\item
$\sigma^d=\sigma^{dd}=\sigma^{cd}=\sigma^{dc}$.
\end{enumerate}
\end{lemma}
\begin{proof}
(1) and (2) are easy.
(3). %
Indeed, for any $\alpha\in(0,1]$ we have:
\begin{multline*}
\sigma^{dd}(\alpha)
=\cup\{\sigma^d(\beta)\mid\;\beta>\alpha\}
=\cup\{\cup\{\sigma(\gamma)\mid\;\gamma>\beta\}\mid\;\beta>\alpha\}\\
=\cup\{\sigma(\beta)\mid\;\beta>\alpha\}
=\sigma^d(\alpha).
\end{multline*}
In the same way, for any $\alpha\in(0,1]$ we have:
\begin{multline*}
\sigma^{cd}(\alpha)
=\cup\{\sigma^c(\beta)\mid\;\beta>\alpha\}
=\cup\{\cup\{\sigma(\gamma)\mid\;\gamma\leq\beta\}\mid\;\beta>\alpha\}\\
=\cup\{\sigma(\beta)\mid\;\beta>\alpha\}
=\sigma^d(\alpha).
\end{multline*}
\end{proof}

A gradual subset $\sigma$ is an \textbf{strict decreasing gradual subset} if $\sigma=\sigma^d$. We have:

\begin{lemma}
For any gradual subset $\sigma$ the following statements hold.
\begin{enumerate}[(1)]\sepa
\item
$\sigma^d$ is the smallest strict decreasing gradual subset contained in $\sigma^c$.
\item
$\sigma$ is a decreasing gradual subset non strict decreasing if, and only if, $\sigma^d\subsetneqq\sigma^c$.
\end{enumerate}
\end{lemma}
\begin{proof}
Let $\tau$ be a strict decreasing gradual subset such that $\tau\subseteq\sigma^c$, then $\tau=\tau^d\subseteq\sigma^{cd}=\sigma^d$.
\end{proof}

This means that the operator $d$ is an interior operator in the set of all decreasing gradual subsets of $X$.

Remember that an \textbf{interior operator} in a poset $P$ is a map $d:P\longrightarrow{P}$ satisfying:
\begin{enumerate}[(1)]\sepa
\item
$d(p)\leq{p}$ for any $p\in{P}$.
\item
For any $p_1,p_2\in{P}$ such that $p_1\leq{p_2}$ we have $d(p_1)\leq{d(p_2)}$.
\item
$d(p)=dd(p)$ for any $p\in{P}$
\end{enumerate}
The elements $p\in{P}$ such that $d(p)=p$ are named the \textbf{$d$--open} elements. Thus, the decreasing gradual subsets open for the operator $d$ are the strict decreasing gradual subsets. Let us denote by $\mathcal{J}^d(X)$ the set of all $d$--open (strict) decreasing gradual subsets.

\begin{remark}\label{re:20180404}
Inspired in these constructions we consider a new construction of a gradual subset from a partial gradual subset that allows us to avoid the {initial} restriction of inf--compact in the domain of definition of partial gradual elements.

Let $\sigma:(0,1]\longrightarrow{P(X)}$ be a partial map defined at 1, i.e., $1\in\dom(\sigma)\subseteq(0,1]$, and such that $\dom(\sigma)$ is not necessarily inf--compact, we may extend $\sigma$ to all $(0,1]$ simply defining $\sigma(\beta)=\varnothing$ if $\beta\notin\dom(\sigma)$. The decreasing gradual subset associated to $\sigma$ is $\sigma^c:(0,1]\longrightarrow{P(X)}$ defined as:
\[
\sigma^c(\alpha)=\cup\{\sigma(\beta)\mid\;\beta\geq\alpha,\,\beta\in\dom(\sigma)\},\textrm{ for any }\alpha\in(0,1].
\]
\end{remark}

The use of decreasing gradual subsets is due to the fact that gradual subsets are wild structures that one can not be managed, and in which there is no relationship between its components. On the other hand, when studying subsets of a given set, it seems natural to impose some inclusion relationships and that these inclusions should be parameterized by the order relation in $(0,1]$.

\begin{remark}
Observe that we may extend any gradual subset $\sigma$ on $X$ to $\overline{\sigma}$ on the whole interval $[0,1]$, in defining $\overline{\sigma}(\alpha)=\left\{
\begin{array}{lll}
\cup\{\sigma(\beta)\mid\;\alpha\leq\beta\in(0,1]\},&\textrm{ if }\alpha\in(0,1]\\
\cup\{\sigma(\beta)\mid\;\beta\in(0,1]\},&\textrm{ if }\alpha=0
\end{array}\right.$ for any $\alpha\in[0,1]$.
In consequence, we may consider also decreasing gradual subsets as maps from $[0,1]$ to $X$.
\end{remark}

\subsection{The algebra of gradual subsets}\label{pg:20181112}

There is a natural relationship between gradual elements and gradual subsets of a given set $X$. Thus, for any partial gradual element $\varepsilon$ we may define a \textbf{unitary partial gradual subset} $\sigma(\varepsilon)$ as $\sigma(\varepsilon)(\alpha)=\{\varepsilon(\alpha)\}$, for any $\alpha\in\dom(\varepsilon)$. As we point out before, we have $\varepsilon\in\sigma$ if, and only if, $\sigma(\varepsilon)\subseteq\sigma$, for any gradual subset $\sigma$.

In the set $\mathcal{P}(X)$ there are two operations: the intersection and the union; thus, we can translate these two operations to gradual subsets, as did in the first section.
Following this line we define, for any gradual subsets, $\sigma_1$ and $\sigma_2$:
\begin{enumerate}[(1)]\sepa
\item
the \textbf{intersection}, $\sigma_1\cap\sigma_2$, as $(\sigma_1\cap\sigma_2)(\alpha)=\sigma_1(\alpha)\cap\sigma_2(\alpha)$, for any $\alpha\in\dom(\sigma_1)\cap\dom(\sigma_2)$,
\item
the \textbf{union}, $\sigma_1\cup\sigma_2$, as $(\sigma_1\cup\sigma_2)(\alpha)=\sigma_1(\alpha)\cup\sigma_2(\alpha)$, for any $\alpha\in\dom(\sigma_1)\cap\dom(\sigma_2)$.
\end{enumerate}

In this way we may consider the algebra of gradual subsets of a given set $X$ with respect to intersection and union.

The definition of intersection and union can also be extended to arbitrary families of gradual subsets. Let $\{\sigma_i\mid\;i\in{I}\}$ be a family of gradual subsets,
\begin{enumerate}[(1)]\sepa
\item
the \textbf{intersection} $\cap_i\sigma_i$, defined as: $(\cap_i\sigma_i)(\alpha)=\cap_i\sigma_i(\alpha)$, for any $\alpha\in(0,1]$.
\item
the \textbf{union} $\cup_i\sigma_i$, defined as: $(\cup_i\sigma_i)(\alpha)=\cup_i\sigma_i(\alpha)$, for any $\alpha\in(0,1]$.
\end{enumerate}

Let $\{\mu_i\mid\;i\in{I}\}$ be a family of fuzzy subsets of a set $X$, the union, $\vee_i\mu_i$, and the intersection, $\wedge_i\mu_i$, are the fuzzy subsets defined by:
\[
\begin{array}{ll}
(\vee_i\mu_i)(a)=\vee_i\mu_i(a)=\Sup\{\mu_i(a)\mid\;i\in{I}\},\textrm{ for any }a\in{X},\\
(\wedge_i\mu_i)(a)=\wedge_i\mu_i(a)=\Inf\{\mu_i(a)\mid\;i\in{I}\},\textrm{ for any }a\in{X}.
\end{array}
\]

{\begin{example}\label{ex:20180313}
Let $X=\{a,b\}$ be a set, for any $n\in\mathbb{N}\setminus\{0,1\}$ we define $\mu_n:\{a,b\}\longrightarrow[0,1]$ by $\mu_n(a)=1$, and $\mu_n(b)=\frac{1}{2}-\frac{1}{2^n}$. We have:
\begin{enumerate}[(1)]\sepa
\item
$(\vee_n\mu_n)(a)=1$ and $(\vee_n\mu_n)(b)=\frac{1}{2}$,
\item
$\sigma(\vee_n\mu_n)(\delta)=\left\{\begin{array}{ll}
X,&\textrm{ if }\delta>\frac{1}{2},\\
\{a\},&\textrm{ if }\delta\leq\frac{1}{2},
\end{array}\right.$
\item
$(\cup_n\sigma(\mu_n))(\delta)=\left\{\begin{array}{ll}
X,&\textrm{ if }\delta\geq\frac{1}{2},\\
\{a\},&\textrm{ if }\delta<\frac{1}{2},
\end{array}\right.$
\end{enumerate}
Which shows that the inclusion $\sigma(\vee_n\mu_n)\supseteq\cup_n\sigma(\mu_n)$ is proper.
\end{example}}

In the same line we have a similar situation for $\widetilde{\sigma}$ and the intersection.

{\begin{example}\label{ex:20180313b}
Let $X=\{a,b\}$ be a set, for any $n\in\mathbb{N}\setminus\{0,1\}$ we define $\mu_n:\{a,b\}\longrightarrow[0,1]$ by $\mu_n(a)=1$, and $\mu_n(b)=\frac{1}{2}+\frac{1}{2^n}$. We have:
\begin{enumerate}[(1)]\sepa
\item
$(\wedge_n\mu_n)(a)=1$ and $(\vee_n\mu_n)(b)=\frac{1}{2}$,
\item
$\widetilde{\sigma}(\wedge_n\mu_n)(\delta)=\left\{\begin{array}{ll}
X,&\textrm{ if }\delta\geq\frac{1}{2},\\
\{a\},&\textrm{ if }\delta<\frac{1}{2},
\end{array}\right.$
\item
$(\cap_n\widetilde{\sigma}(\mu_n))(\delta)=\left\{\begin{array}{ll}
X,&\textrm{ if }\delta>\frac{1}{2},\\
\{a\},&\textrm{ if }\delta\leq\frac{1}{2},
\end{array}\right.$
\end{enumerate}
Which shows that the inclusion $\widetilde{\sigma}(\wedge_n\mu_n)\subseteq\cap_n\widetilde{\sigma}(\mu_n)$ is proper.
\end{example}}

In the set $\mathcal{J}^d(X)$ of all strict decreasing gradual subsets of $X$ we define two new operations: intersection:
$\underline{\wedge}_i\sigma_i=(\cap_i\sigma_i)^d$, and maintain the old union: $\underline{\vee}_i\sigma_i=\cap_i\sigma_i$, for every family $\{\sigma_i\mid\;i\in{I}\}$ of strict decreasing gradual subsets of $X$. With these definition we have:

\begin{proposition}\label{pr:20180416}
The union and intersection of strict decreasing gradual subsets are compatible with the union and intersection of fuzzy subsets via the gradual subset $\widetilde{\sigma}(\mu)$, i.e., for any family of fuzzy subsets $\{\mu_i\mid\;i\in{I}\}$ we have:
\begin{enumerate}[(1)]\sepa
\item
$\underline{\vee}_i\widetilde{\sigma}(\mu_i)=\widetilde{\sigma}(\vee_i\mu_i)$.
\item
$\underline{\wedge}_i\widetilde{\sigma}(\mu_i)=\widetilde{\sigma}(\wedge_i\mu_i)$.
\end{enumerate}
\end{proposition}
\begin{proof}
For any $\alpha\in(0,1]$ we have:
\[
\begin{array}{ll}
\widetilde{\sigma}(\vee_i\mu_i)(\alpha)
&=\{a\in{X}\mid\;(\vee_i\mu_i)(a)>\alpha\}\\
&=\{a\in{X}\mid\;\vee_i\mu_i(a)>\alpha\}\\
&=\cup\{a\in{X}\mid\;\textrm{there exists $i$ such that }\mu_i(a)>\alpha\}\\
&=(\cup_i\widetilde{\sigma}(\mu_i))(\alpha)=(\underline{\vee}_1\widetilde{\sigma}(\mu_i))(\alpha).
\end{array}
\]
In the same way we can prove the case of $\widetilde{\sigma}(\wedge_i\mu_i)(\alpha)$.
\end{proof}

From this point of view strict $\alpha$--levels should be a suitable tool for studying the algebra of fuzzy subsets via decreasing gradual subsets.

\subsection{Maps}

In order to relate two gradual subsets, a standard method consists in defining a map from one to the other. In this context first we consider a map between the underlying sets containing each gradual subset; the following result show how to associate gradual subsets to gradual subsets via a map.

\begin{lemma}[Direct image]
Let $f:X\longrightarrow{Y}$ be a map, and denote by $f$ the induced map from $\mathcal{P}(X)$ to $\mathcal{P}(Y)$, the following statements hold:
\begin{enumerate}[(1)]\sepa
\item
For every gradual element $\varepsilon$ of $X$ we have that $f\varepsilon:\dom(\varepsilon)\longrightarrow{Y}$ is a gradual element of $Y$.
\item
Let $\sigma$ be a gradual subset of $X$, then $f\sigma:\dom(\sigma)\longrightarrow\mathcal{P}(Y)$ is a gradual subset of $Y$. And we have a map $f_*:\mathcal{X}\longrightarrow\mathcal{Y}$ defined $f_*(\sigma)=f\sigma$ for any $\sigma\in\mathcal{X}$.
\item
Let $\sigma$ be a gradual subset of $X$, then $f_*(\sigma^c)=(f_*(\sigma))^c$.
\end{enumerate}
In addition, if $\varepsilon\in\sigma$, then $f\varepsilon\in{f\sigma}$.
\end{lemma}

\begin{lemma}[Inverse image]
Let $f:X\longrightarrow{Y}$ be a map, and denote by $f^{-1}:\mathcal{P}(Y)\longrightarrow\mathcal{P}(X)$ the induced map. For any gradual subset $\tau$ of $Y$ we have that $f^*\tau:\dom(\tau)\longrightarrow\mathcal{P}(X)$, defined as $f^*\tau(\alpha)=f^{-1}(\tau(\alpha))$, for any $\alpha\in\dom(\tau)$, is a partial gradual subset of $X$. Thus, we have a map $f^*:\mathcal{Y}\longrightarrow\mathcal{X}$, defined $f^*(\tau)=f^*\tau$, for any $\tau\in\mathcal{Y}$.
\par
In particular, for any gradual subset $\tau$ of $Y$, we have: $f^*(\tau^c)=(f^*(\tau))^c$.
\end{lemma}

Since every element of $X$ and every element of $Y$ are gradual elements, and the same for gradual subsets, the notions of injective map and surjective map, applied either to gradual elements or gradual subsets are equivalent. In the case of gradual subsets we have:

\begin{lemma}
Let $f:X\longrightarrow{Y}$ be a map, then:
\begin{enumerate}[(1)]\sepa
\item
$f$ is injective if, and only if, {$f^*\circ{f_*}=\id_{\mathcal{X}}$}.
\item
$f$ is surjective if, and only if, {$f_*\circ{f^*}=\id_{\mathcal{Y}}$}.
\end{enumerate}
\end{lemma}

Our aim will be to establish maps between gradual sets instead of between gradual subsets, i.e., leave out the ground set $X$ and use only the subsets $\{\sigma(\alpha)\mid\;\alpha\in(0,1]\}$. But we postpone it until the point in which we change the paradigm introducing these gradual sets.

\begin{remark}\label{re:2018110907}
A gradual subset $\sigma$ of a set $X$ is just a family $\{\sigma(\alpha)\mid\;\alpha\in(0,1]\}$ of subsets, indexed in $(0,1]$. There are particular types of gradual subsets, as decreasing gradual subsets, in which, for any $\alpha,\beta\in(0,1]$, $\alpha\leq\beta$, there exists a map $j_{\alpha,\beta}:\sigma(\beta)\longrightarrow\sigma(\alpha)$: the inclusion, satisfying $j_{\alpha,\beta}j_{\beta,\gamma}=j_{\alpha,\gamma}$ whenever $\alpha\leq\beta\leq\gamma$. In some sense, decreasing gradual subsets are gradual subsets enriched with a family of maps $\{j_{\alpha,\beta}\mid\;\alpha,\beta\in(0,1],\alpha\leq\beta\}$ satisfying the above conditions and compatible with the inclusions in $X$. Thus, we may work with these enriched gradual subsets of $X$.

An \textbf{enriched gradual subsets} of $X$ is a gradual subset $\sigma$ together with a family of maps $\{f_{\alpha,\beta}\mid\;\alpha,\beta\in(0,1],\alpha\leq\beta\}$ satisfying:
\begin{enumerate}[(1)]\sepa
\item
$f_{\alpha,\beta}:\sigma(\beta)\longrightarrow\sigma(\alpha)$,
\item
$f_{\alpha,\beta}f_{\beta,\gamma}=f_{\alpha,\gamma}$ whenever $\alpha\leq\beta\leq\gamma$,
\item\label{it:20181109-3}
if $j_\alpha:\sigma(\alpha)\longrightarrow{X}$ is the inclusion, for any $\alpha\in(0,1]$, then $j_\alpha{f_{\alpha,\beta}}=j_\beta$, whenever $\alpha\leq\beta$.
\end{enumerate}
Observe that, as a consequence of \eqref{it:20181109-3}, each $f_{\alpha,\beta}$ is an injective map. In particular, enriched gradual subsets are just the decreasing gradual subsets. See also Remark~\eqref{re:20181109}.
\end{remark}

\subsection{{Gradual} quotient sets}

The same technique we used to introduce gradual subsets can be applied to define quotient gradual sets of a given set $X$.

Remember that if $X$ is a set, a \textbf{subset} $S\subseteq{X}$ is an equivalence class in the class of all injective maps $\{(i,Y)\mid\;i:Y\longrightarrow{X}\textrm{ injective}\}$, whenever we consider the equivalence relation: $(i_1,Y_1)\sim(i_2,Y_2)$ if there exists a bijective map $b:Y_1\longrightarrow{Y_2}$ such that $i_1=i_2b$.
\[
\begin{xy}
\xymatrix{
Y_1\ar[rrd]^{i_1}\ar[d]^b\\
Y_2\ar[rr]_{i_2}&&X
}\end{xy}
\]
Dually, a \textbf{quotient set} of $X$ is an equivalence class in the class of all surjective maps $\{(Z,p)\mid\;p:X\longrightarrow{Z}\textrm{ surjective}\}$ when we consider the equivalence relation: $(Z_1,p_1)\sim(Z_2,p_2)$ if there exists a bijective map $b:Z_1\longrightarrow{Z_2}$ such that $p_2=bp_1$.
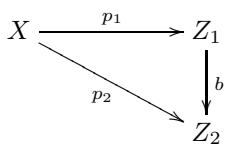
\[
\begin{xy}
\xymatrix{
X\ar[rr]^{p_1}\ar[rrd]_{p_2}&&Z_1\ar[d]^b\\
&&Z_2
}\end{xy}
\]

The set of all subsets of $X$ is represented by $\mathcal{P}(X)$, and there exists a bijective correspondence between $\mathcal{P}(X)$ and $2^X$. The set of all quotient set of $X$ will be represented by $\mathcal{Q}(X)$, and for any element $Z\in\mathcal{Q}(X)$ we have:
\begin{enumerate}[(1)]\sepa
\item
a surjective map $p:X\longrightarrow{Z}$,
\item
an equivalence relation $R_p$ in $X$ defined as $x{R_p}y$ if $p(x)=p(y)$, and
\item
a partition of $X$ into the equivalence classes defined by a relation $R$.
\end{enumerate}

Each equivalence relation $R$ in $X$ is a subset of $X\times{X}$ satisfying the properties reflexive, symmetric and transitive. Hence $\mathcal{Q}(X)$ is in bijection with a subset of $\mathcal{P}(X\times{X})$. If we call $\mathcal{Q}(X\times{X})$ this subset, it is constituted by all the equivalence relations in $X$.

A \textbf{gradual quotient set} of $X$ is a gradual element of $\mathcal{Q}(X)$, or equivalently, of $\mathcal{Q}(X\times{X})$. We represent by $\rho$ a gradual quotient set of $X$.

\subsection{Gradual subsets and fuzzy subsets}

As an example of application of the gradual subset theory let us to establish a correspondence between fuzzy subsets and enriched gradual subsets. As we had shown before, see Proposition~\eqref{pr:20180416}, if we consider the strict decreasing gradual subset $\widetilde{\sigma}(\mu)$, the correspondence $\mu\mapsto\widetilde{\sigma}(\mu)$ is a homomorphism with respect to arbitrary union and intersection.

In addition, the gradual subsets $\sigma(\mu)$ and $\widetilde{\sigma}(\mu)$ are related in a strong way: $\widetilde{\sigma}(\mu)=\sigma(\mu)^d\subseteq\sigma(\mu)^c=\sigma(\mu)$, using the interior and closure operator. Also, these gradual subsets satisfy the following properties:
\begin{enumerate}[(1)]\sepa\label{pa:20180417}
\item
$\mu(x)=\Max\{\alpha\mid\;x\in\sigma(\mu)(\alpha)\}$, for any $x\in{X}$, and
\item
$\mu(x)=\Inf\{\alpha\mid\;x\notin\widetilde{\sigma}(\mu)(\alpha)\}$, for any $x\in{X}$.
\end{enumerate}

We say
\begin{enumerate}[(1)]\sepa
\item
a decreasing gradual subset $\sigma$ satisfies \textbf{property (F)} if there exists $\Max\{\alpha\mid\;x\in\sigma(\alpha)\}$ for every $x\in\cup\{\alpha\mid\;\alpha\in(0,1]\}$, and
\item
a strict decreasing gradual subset $\sigma$ satisfies satisfies \textbf{property (inf--F)} if $\Inf\{\alpha\mid\;x\notin\sigma^d(\alpha)\}=\beta$ satisfies $x\in\sigma(\beta)$, for every $x\in\cup\{\alpha\mid\;\alpha\in(0,1]\}$.
\end{enumerate}

As a consequence we have the following result:

\begin{lemma}
Let $\sigma$ be a decreasing gradual subset, not strict decreasing gradual subset, the following statements are equivalent:
\begin{enumerate}[(a)]\sepa
\item
$\sigma$ satisfies property (F).
\item
$\sigma^d$ satisfies property (inf--F).
\item
$\cup_{\alpha\in(0,1]}\sigma(\alpha)=\stackrel{\bullet}{\cup}_{\alpha\in(0,1]}(\sigma^c(\alpha)\setminus\sigma^d(\alpha))$ (the disjoint union).
\end{enumerate}
\end{lemma}
\begin{proof}
By the hypothesis we have $\sigma^d\subsetneqq\sigma^c=\sigma$.
\par (a) $\Rightarrow$ (b). %
Let $\beta=\Max\{\alpha\mid\;x\in\sigma(\alpha)\}$ and $\gamma=\Inf\{\alpha\mid\;x\notin\sigma^d(\alpha)\}$, then $x\notin\sigma(\delta)$ for every $\delta>\gamma$. Since $\sigma$ is decreasing, $\beta\leq\gamma$. If $\beta<\gamma$, for any $\omega$ such that $\beta<\omega<\gamma$ we have $x\notin\sigma(\omega)$, hence $\gamma\neq\Inf\{\alpha\mid\;x\notin\sigma^d(\alpha)\}$, which is a contradiction.
\par (b) $\Rightarrow$ (a). %
Let $\beta=\Sup\{\alpha\mid\;x\in\sigma(\alpha)\}$ and $\gamma=\Inf\{\alpha\mid\;x\notin\sigma^d(\alpha)\}$. Since $\sigma$ is decreasing, $\beta\leq\gamma$. If $\beta<\gamma$, for any $\omega$ such that $\beta<\omega<\gamma$ we have $x\in\sigma(\omega)$, and $\beta\neq\Sup\{\alpha\mid\;x\in\sigma(\alpha)\}$, which is a contradiction.
\par (a) $\Rightarrow$ (c). %
One inclusion is obvious. Otherwise, if $x\in\cup_{\alpha\in(0,1)}\sigma(\alpha)$, let $\beta=\Max\{\alpha\mid\;x\in\sigma(\alpha)\}$, then $x\notin\sigma(\gamma)$ for any $\gamma>\beta$, hence $x\notin\sigma^d(\beta)$, and $x\in\stackrel{\bullet}{\cup}_{\alpha\in(0,1]}(\sigma^c(\alpha)\setminus\sigma^d(\alpha))$.
\par (c) $\Rightarrow$ (a). %
Let $x\in{X}$, if $x\notin\cup_{\alpha\in(0,1)}\sigma(\alpha)$, then either $x\in\sigma(1)$, and there exists $\Max\{\alpha\mid\;x\in\sigma(\alpha)\}=1$, or $x\notin\cup_{\alpha\in(0,1]}\sigma(\alpha)$, and $\Max\{\alpha\mid\;x\in\sigma(\alpha)\}=0$. Otherwise, if $x\in\cup_{\alpha\in(0,1)}\sigma(\alpha)=\stackrel{\bullet}{\cup}_{\alpha\in(0,1]}(\sigma^c(\alpha)\setminus\sigma^d(\alpha))$, there exists $\alpha$ such that $x\in\sigma(\alpha)\setminus\sigma^d(\alpha)$, hence $\Max\{\alpha\mid\;x\in\sigma(\alpha)\}=\alpha$.
\end{proof}

\begin{remark}\label{re:20180417}
\begin{enumerate}[(1)]%\sepa
\item
As a consequence of this result, for any decreasing non strict decreasing gradual subset, satisfying property (F), we have $\Max\{\alpha\mid\;x\in\sigma(\alpha)\}=\Inf\{\alpha\mid\;x\notin\sigma^d(\alpha)\}$ for any $x\in{X}$.
\item
In the case of a strict decreasing gradual subset satisfying property (inf--F) we have that it also satisfies property (F), hence the following equalities hold: $\Max\{\alpha\mid\;x\in\sigma(\alpha)\}=\Inf\{\alpha\mid\;x\notin\sigma^d(\alpha)\}$ for any $x\in{X}$.
\end{enumerate}
\end{remark}

Now we are going to establish correspondences between fuzzy subsets and strict decreasing gradual subsets, which preserves union and intersection. The following result, for finite unions, is well known.

\begin{theorem}
Let $X$ be a set, the following statements hold:
\begin{enumerate}[(1)]\sepa
\item
The map $\nu:\mu\mapsto\sigma(\mu)$ associates to any fuzzy subset $\mu$ of $X$ a decreasing gradual subset satisfying property (F), $\sigma(\mu)$ of $X$, defined $\sigma(\mu)(\alpha)=\{x\mid\;\mu(x)\geq\alpha\}$, and it preserves intersections and finite unions.
\item
The map $\upsilon:\sigma\mapsto\mu(\sigma)$ associates to any decreasing gradual subset $\sigma=\sigma^c$, satisfying property (F), a fuzzy subset $\mu(\sigma)$ defined:
\[
\mu(\sigma)(x)=\Max\{\alpha\mid\;x\in\sigma(\alpha)\}.
\]
\end{enumerate}
In addition, we have {$\nu\circ\upsilon=\id$} and {$\upsilon\circ\nu=\id$}, and they preserve finite unions and intersections.
\end{theorem}

The behaviour with respect to infinite unions can be solved using only strict decreasing gradual subsets instead of decreasing gradual subsets.

\begin{theorem}\label{th:20180417}
Let $X$ be a set, the following statements hold:
\begin{enumerate}[(1)]\sepa
\item
The map $\nu:\mu\mapsto\widetilde{\sigma}(\mu)$ associates to any fuzzy subset $\mu$ of $X$ a strict decreasing gradual subset $\widetilde{\sigma}(\mu)$ of $X$, defined $\widetilde{\sigma}(\mu)(\alpha)=\{x\mid\;\mu(x)>\alpha\}$, and it preserves arbitrary unions and intersections.
\item
The map $\upsilon:\sigma\mapsto\widetilde{\mu}(\sigma)$ associates to any strict decreasing gradual subset $\sigma=\sigma^d$, satisfying property (inf--F), a fuzzy subset $\widetilde{\mu}(\sigma)$ defined:
\[
\widetilde{\mu}(\sigma)(x)=\Inf\{\alpha\mid\;x\notin\sigma(\alpha)\}.
\]
\end{enumerate}
In addition, we have {$\nu\upsilon=\id$}, and {$\upsilon\nu=\id$}.
\end{theorem}
\begin{proof}
We had already studied the map $\nu$ in Proposition~\eqref{pr:20180416} in page \pageref{pa:20180417}.
\par
The map $\upsilon$ is well defined due to Theorem~\eqref{th:20180417}. Now, we check that the compositions are the identity.
\par
Let $\mu$ be a fuzzy subset of $X$, for any $x\in{X}$, we have:
\[
{\upsilon\circ\nu(\mu)(x)}
=\widetilde{\mu}(\widetilde{\sigma}(\mu))x)
=\Inf\{\alpha\mid\;x\notin\widetilde{\sigma}(\mu)(\alpha)\}
=\Inf\{\alpha\mid\;\mu(x)\leq\alpha\}
=\mu(x).
\]
On the other hand, let $\sigma$ be a strict decreasing gradual subset, and $\alpha\in(0,1]$, we have:
\[
{\nu\circ\upsilon(\sigma)(\alpha)}
=\widetilde{\sigma}(\widetilde{\mu}(\sigma))(\alpha)
=\{x\mid\;\widetilde{\mu}(\sigma)(x)>\alpha\}
=\{x\mid\;\Inf\{\beta\mid\;x\notin\sigma(\beta)\}>\alpha\}.
\]
If $\Inf\{\beta\mid\;x\notin\sigma(\beta)\}=\delta_x$, then $\delta_x>\alpha$, since $\sigma$ satisfies property (inf--F), then $x\in\sigma(\delta_x)\subseteq\sigma(\alpha)$. Otherwise, if $x\in\sigma(\alpha)=\sigma^d(\alpha)=\cup\{\sigma(\beta)\mid\;\beta>\alpha\}$, there exists $\gamma>\alpha$ such that $x\in\sigma(\gamma)$, hence $\Inf\{\beta\mid\;x\notin\sigma(\beta)\}>\gamma>\alpha$, and we have the other inclusion.
\end{proof}

As a consequence, we have the final result that establishes an isomorphism between the two lattices. See Proposition~\eqref{pr:20180416}.

\begin{corollary}
Let $X$ be a set, there is an isomorphism between the lattice of all fuzzy subset of $X$ and the lattice of all strict decreasing gradual subset of $X$, satisfying property (F), and they preserves arbitrary unions and intersections.
\end{corollary}

\subsection{A functorial interpretation}

Let us consider $(0,1]$ as a category whose objects are the elements of $(0,1]$, and homomorphisms: only one, $f_{\alpha,\beta}$, from $\alpha$ to $\beta$ whenever $\alpha\leq\beta$, and the obvious composition.

Let $F:(0,1]\longrightarrow\Set$ a contravariant functor from the category $(0,1]$ to the category of sets. Indeed, $\{F(\alpha)\mid\;\alpha\in(0,1]\}$ is a directed system with maps $F(f_{\alpha,\beta}):F(\beta)\longrightarrow{F(\alpha)}$, if $\alpha\leq\beta$. Let $D=\dlim{F}$ be the direct limit of this system.

Let us remember the definition of the {\textbf{direct limit}}, $D=\dlim{F}$. First we consider the disjoint union, $\stackrel{\bullet}{\cup}F(\alpha)$, of the family of sets $\{F(\alpha)\mid\;\alpha\in(0,1]\}$, and in it, the equivalence relation $R$ generated by: $a\in{F(\alpha)}$ is \emph{related with} $b\in{F(\beta)}$ if either $\alpha\leq\beta$ and $F(f_{\alpha,\beta})(b)=a$ or $\beta\leq\alpha$ and $F(f_{\beta,\alpha})(a)=b$. Let $p:\stackrel{\bullet}{\cup}F(\alpha)\longrightarrow(\stackrel{\bullet}{\cup}F(\alpha))/R$ be the canonical projection, $i_\beta:F(\beta)\longrightarrow\stackrel{\bullet}{\cup}F(\alpha)$ be the inclusion, for any $\beta\in(0,1]$, and $q_\beta=p\,i_\beta$ the composition.
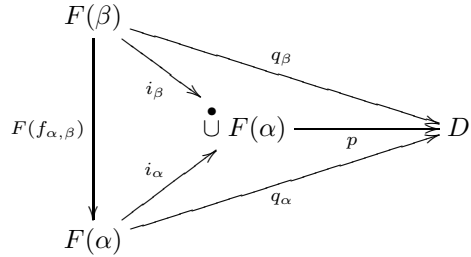
\[
\begin{xy}
\xymatrix{
F(\beta)\ar[dr]_{i_\beta}\ar[rrrd]^{q_\beta}\ar[dd]_{F(f_{\alpha,\beta})}\\
&\stackrel{\bullet}{\cup}F(\alpha)\ar[rr]_p&&D\\
F(\alpha)\ar[ru]^{i_\alpha}\ar[rrru]_{q_\alpha}
}\end{xy}
\]
The pair $(D,\{p_\alpha\mid\;\alpha\in(0,1]\})$ satisfies the corresponding universal property of the direct limit.

In the particular case in which every map $F(f_{\alpha,\beta})$ is injective, then every $q_\alpha$ is also injective; this means that we can consider every $F(\alpha)$ as a subset of $D$. Therefore, we have that $F$ defines a decreasing gradual subset of $D=\dlim{F}$; or more generally, a decreasing gradual subset of any overset of $D$, and represent it by $(F,\dlim{F})$.

This allows to give an interpretation of decreasing gradual subsets in terms of contravariant functors. If we start from a decreasing gradual subset of a set $X$, then $\{\sigma(\alpha)\mid\;\alpha\in(0,1]\}$, together with the family of inclusions, is a directed system, and if $j_\alpha:\sigma(\alpha)\longrightarrow{X}$, for every $\alpha\in(0,1]$, is the inclusion, then we have a commutative diagram
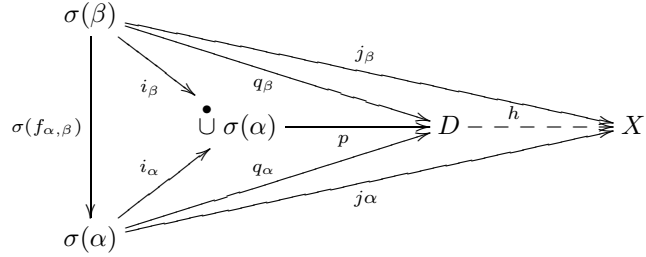
\[
\begin{xy}
\xymatrix{
\sigma(\beta)\ar[dr]_{i_\beta}\ar[rrrd]_{q_\beta}\ar[dd]_{\sigma(f_{\alpha,\beta})}\ar[rrrrrd]^{j_\beta}\\
&\stackrel{\bullet}{\cup}\sigma(\alpha)\ar[rr]_p&&D\ar@{-->}[rr]^(.35)h&&X\\
\sigma(\alpha)\ar[ru]^{i_\alpha}\ar[rrru]^{q_\alpha}\ar[rrrrru]_{j\alpha}
}\end{xy}
\]
and an inclusion $h:D\longrightarrow{X}$, from the direct limit $D$ to $X$, beings $h(D)=\Ima(h)$ the union of the family of the subsets $\{\sigma(\alpha)\mid\;\alpha\in(0,1]\}$.

Taking into account this construction, we may show that a {decreasing} gradual subset $\sigma$ of a set $X$ is nothing more than a {contravariant} functor $F:(0,1]\longrightarrow\Set$ of this shape, together with an injective map $D=\dlim{F}\longrightarrow{X}$. As a consequence, we may consider {contravariant} functors as the central element in the study of decreasing gradual subsets.

Hence, we may define a \textbf{gradual set} as a contravariant functor $F:(0,1]\longrightarrow\Set$, {and a \textbf{decreasing gradual set} as a gradual set such that each map} $F(f_{\alpha,\beta})$, whenever $\alpha\leq\beta$, is injective.

Given two gradual sets $F_1$ and $F_2$, a map from $F_1$ to $F_2$ is just a natural transformation $\theta:F_1\longrightarrow{F_2}$, i.e., a set of maps $\{\theta_\alpha\mid\;\alpha\in(0,1]\}$ such that each diagram commutes, whenever $\alpha\leq\beta$.
\[
\begin{xy}
\xymatrix{
F_1(\beta)\ar[rr]^{F_1(f_{\alpha,\beta})}\ar[d]_{\theta_\beta}&&F_1(\alpha)\ar[d]^{\theta_\alpha}\\
F_2(\beta)\ar[rr]^{F_2(f_{\alpha,\beta})}&&F_2(\alpha)\\
}\end{xy}
\]
The contravariant functors from $(0,1]$ to $\Set$, i.e., the gradual sets, constitute a category that we shall denote by $\Set^{(0,1]}$. {The class of all decreasing gradual sets defines a full subcategory of $\Set^{(0,1]}$, and it is closed under (finite and infinite) unions and intersections. Let us call $\mathcal{J}$ this subcategory.

{In the subcategory $\mathcal{J}$ we define an interior operator,} $d:F\mapsto{F^d}$, as follows:
\[
F^d(\alpha)=\cup\{F(\gamma)\mid\;\gamma>\alpha\}=\dlim_{(\alpha,1]}F(\gamma),
\]
and if $\alpha\leq\beta$, then there is an inclusion functor $(\beta,1]\longrightarrow(\alpha,1]$, and a natural map from $F^d(\beta)=\dlim_{(\beta,1]}F(\gamma)$ to $F^d(\alpha)=\dlim_{(\alpha,1]}F(\gamma)$.

\begin{proposition}
Let $F$ be a decreasing gradual set, and $\theta:F_1\longrightarrow{F_2}$ be a decreasing gradual set map. The following statements hold.
\begin{enumerate}[(1)]\sepa
\item
$F^d$ is a contravariant functor from $(0,1]$ to $\Set$, hence it is a gradual set.
\item
If $\alpha\leq\beta$, the natural map $F^d(\beta)\longrightarrow{F^d(\alpha)}$ is injective, hence $F^d$ is a decreasing gradual set.
\item
There exists a natural map $\theta^d:F_1^d\longrightarrow{F_2^d}$, defined, for every $\beta\in(0,1]$, as the only map making commutative the following diagram
\[
\begin{xy}
\xymatrix{
F_1(\gamma)\ar[rr]\ar[d]_{\theta_\beta}&&\dlim_{(\beta,1]}F_1(\gamma)=F_1^d(\beta)\ar[rr]\ar@{-->}[d]^{\theta^d(\beta)}&&\dlim_{(0,1]}F_1(\gamma)\ar[d]\\
F_2(\gamma)\ar[rr]&&\dlim_{(\beta,1]}F_2(\gamma)=F_1^d(\beta)\ar[rr]&&\dlim_{(0,1]}F_2(\gamma)
}\end{xy}
\]
\item
$d$ defines an endofunctor of $\mathcal{J}$, which is an interior operator in $\mathcal{J}$.
\end{enumerate}
\end{proposition}

A decreasing gradual set $F$ is an \textbf{strict decreasing gradual set} whenever $F=F^d$, and satisfies \textbf{property (F)} if $D=\dlim{F}=\stackrel{\bullet}{\cup}\{F(\alpha)\setminus{F^d(\alpha)}\}$, where the union is taken in $D$.

By the relationship between (inf--F) and (F) properties, we may define a fuzzy set as a strict decreasing gradual set satisfying property (inf--F). In particular, strict decreasing gradual sets satisfying property (inf--F) constitute a full subcategory of $\mathcal{J}$.

As a consequence, we have the following result.

\begin{theorem}
Let $F$ be a gradual set,
\begin{enumerate}[(1)]\sepa
\item
The following statements are equivalent:
 \begin{enumerate}[(a)]%\sepa
 \item
 $F$ is a strict decreasing gradual set.
 \item
 The pair $(F,\dlim{F})$ is a strict decreasing gradual subset of $D=\dlim{F}$.
 \end{enumerate}
\item
The following statements are equivalent:
 \begin{enumerate}[(a)]%\sepa
 \item
 $F$ is a strict decreasing gradual set satisfying property (inf--F), i.e., $F$ is a fuzzy set.
 \item
 The pair $(F,\dlim{F})$ is a strict decreasing gradual subset of $D=\dlim{F}$ satisfying property (inf--F).
 \end{enumerate}
\end{enumerate}
\end{theorem}

\begin{remark}
The use of decreasing gradual sets allows to avoid the use of decreasing gradual subsets. Indeed, a decreasing gradual subset is, in some sense, more natural: we can build the category of decreasing gradual sets as a subcategory of the functor category $\Set^{(0,1]}$. Otherwise, decreasing gradual subsets are referenced to a set, the same does not happen with decreasing gradual sets; although, as we have the direct limit, the direct system itself acts as a real set. With fuzzy subsets we have the same situation. Observe that in the fuzzy situation, when we consider the directed systems and the direct limit, we are considering only those elements with a positive, non--zero, membership degree, i.e., we do not consider those with zero membership degree. See also Remark~\eqref{re:2018110907}.
\end{remark}

{\begin{remark}\label{re:20181109}
In looking for an abstract model for gradual subsets of a set $X$, our first candidate was the functor category $\Set^{(0,1]}$. But unfortunately, with this category we do not obtain faithful representation of all gradual subsets. One may consider the gradual subset $\sigma$ of a non--empty set $X$ defined by $\sigma(\delta)=\left\{\begin{array}{ll}
X,&\textrm{ if }\delta=\frac{1}{2},\\
\varnothing,&\textrm{ if }\delta\neq\frac{1}{2}.
\end{array}\right.$
Obviously, we can not obtain $\sigma$ using contravariant functors from the category $(0,1]$: the reason is that there are no maps from $X$ to $\varnothing$ (it is not an enriched gradual subset). This have been overcome when we consider enriched gradual sets. This model works perfectly and meets all our expectations whenever we consider decreasing gradual subsets.
\end{remark}}

{\begin{remark}
Observe that in our construction we have fixed the categories $\Set$ and $(0,1]$, and considered contravariant functor. If we change contravariant for covariant, we get increasing gradual sets. On the other hand, the category $\Set$ has some peculiarities: one is that there $\varnothing$, which is an initial and not a final object; the other is that there are objects $A$ and $B$ such that $\Hom_{\Set}(A,B)=\varnothing$; these forces the use of increasing or decreasing gradual sets to assure of writing the theory in a functor language using the usual order relation in $(0,1]$. Some of these restriction will be removed once we change the category $\Set$ for another category as $\Gr$ (the category of groups) or $\rMod{A}$ (the category of right $A$--modules).
\end{remark}}

\section{Gradual subgroups}

In section~\eqref{pg:20181112} we have studied gradual subsets of $\mathcal{P}(X)$ for any set $X$, and considered in $\mathcal{P}(X)$ the operations: intersections and union. We can repeat the same procedure whenever we have a binary operation in $X$, and translate it into $\mathcal{P}(X)$, or a subset of $\mathcal{P}(X)$, in the natural way. Thus, our aim in this section is to study gradual subsets of a given set $X$, together with an additional algebraic structure in $X$; to do that we shall consider the simplest example of groups.

\subsection{Gradual subgroups}

Let $X$ be a non--empty set with a binary operation $*$, we define in $\mathcal{P}(X)\setminus\{\varnothing\}$ new binary and unary operations by:
\[
\begin{array}{ll}
S_1*S_2=\{s_1*s_2\mid\;s_i\in{S_i}\},&\textrm{ for every }S_1,S_2\in\mathcal{P}(X)\setminus\{\varnothing\}\textrm{ and}\\
S^{-1}=\{s^{-1}\mid\;s\in{S}\},&\textrm{ for every }S\in\mathcal{P}(X)\setminus\{\varnothing\}.
\end{array}
\]
Thus, we may define an operation on gradual subsets of $X$ (for simplicity, in this section, for a set $X$ a gradual subset of $X$ is a gradual element of $\mathcal{P}(X)\setminus\{\varnothing\}$) by:
\[
\begin{array}{ll}
(\sigma_1*\sigma_2)(\alpha)=\sigma_1(\alpha)*\sigma_2(\alpha),&\textrm{ for every }\sigma_1,\sigma_2\textrm{ and }\alpha\in(0,1].\\
\sigma^{-1}(\alpha)=\sigma(\alpha)^{-1},&\textrm{ for every }\sigma\textrm{ and }\alpha\in(0,1].
\end{array}
\]

\begin{definition}
Let $G$ be a group (we eliminate the symbol $*$, and represent the product just as juxtaposition), a \textbf{gradual subgroup} of $G$ is a gradual subset $\sigma$ of $G$, satisfying:
\begin{enumerate}[(1)]\sepa
\item
$\sigma*\sigma\subseteq\sigma$,
\item
$\sigma^{-1}\subseteq\sigma$.
\end{enumerate}
\end{definition}

\begin{proposition}\label{pr:20181112}
Let $e$ be the neutral element of a group $G$, and $\sigma$ a gradual subgroup, the following statements hold:
\begin{enumerate}[(1)]\sepa
\item
$e\in\sigma(\alpha)$ for any $\alpha\in(0,1]$.
\item
$\sigma(\alpha)$ is a subgroup of $G$ for any $\alpha\in(0,1]$.
\end{enumerate}
Therefore, if $S(G)$ is the set of all subgroups of $G$, a gradual subgroup of $G$ is just a gradual element of $S(G)$.
\end{proposition}
\begin{proof}
(1). %
Let $a\in\sigma(\alpha)$ then $a^{-1}\in\sigma(\alpha)$, hence $e=a*a^{-1}\in\sigma(\alpha)$.
\par (2). %
Let $a,b\in\sigma(\alpha)$ then $b^{-1}\in\sigma(\alpha)$, hence $a*b^{-1}\in\sigma(\alpha)$.
\end{proof}

If $\varepsilon$ is a gradual element of a group $G$, for any $\alpha\in(0,1]$ we define $\langle\varepsilon\rangle(\alpha)=\langle\varepsilon(\alpha)\rangle$, the gradual subgroup of $G$ \textbf{generated} by $\varepsilon$. A gradual subgroup $\sigma$ of $G$ is \textbf{cyclic} if there exists a gradual element $\varepsilon$ such that $\sigma=\langle\varepsilon\rangle$.

We may also define finitely generated gradual subgroups: a gradual subgroup $\sigma$ is \textbf{finitely generated} if there are gradual elements $\varepsilon_1,\ldots,\varepsilon_t$ such that for any $\alpha\in{L}$ we have $\sigma(\alpha)=\langle\varepsilon_1(\alpha),\ldots,\varepsilon_t(\alpha)\rangle$. We represent this $\sigma$ simply as $\langle\varepsilon_1,\ldots,\varepsilon_t\rangle$.

\begin{proposition}
Let $\sigma$ be a gradual subgroup of a group $G$. The following statements are equivalent:
\begin{enumerate}[(a)]\sepa
\item
$\sigma$ is finitely generated.
\item
There exist gradual elements $\varepsilon_1,\ldots,\varepsilon_t$ such that $\sigma=\langle\varepsilon_1,\ldots,\varepsilon_t\rangle$.
\item
There exists a positive integer $t$ such that each subgroup $\sigma(\alpha)$ can be generated by $t$ elements.
\end{enumerate}
\end{proposition}

Observe that due to Proposition~\eqref{pr:20181112}, gradual subgroups of $G$ can be identify with subgroups of the direct product, indexed in $(0,1]$, of copies of $G$.

\subsection{Normal gradual subgroups}

By the afore mentioned identification the study of gradual subgroups is very simple. Thus, a gradual subgroup of a group $G$ is \textbf{normal} if for any $\alpha\in(0,1]$ we have that $\sigma(\alpha)\subseteq{G}$ is a normal subgroup of $G$. If $\sigma$ is a normal gradual subgroup of $G$, for every $\alpha\in(0,1]$ we have a quotient group $G/\sigma(\alpha)$.

Let $\sigma$ be a normal gradual subgroup of $G$, for every $\alpha\in(0,1]$ we have a gradual quotient group $G/\sigma(\alpha)$, hence a gradual quotient set $\rho$ of $G$, defined as $\rho(\alpha)=G/\sigma(\alpha)$, for every $\alpha\in(0,1]$. We may represent it also by $G/\sigma$. For every $\alpha\in(0,1]$ there is a group homomorphism $G\longrightarrow{G/\sigma(\alpha)}=(G/\sigma)(\alpha)$.

We define a \textbf{gradual quotient group} of $G$ a gradual quotient set $\eta$ of $G$ such that for every $\alpha\in(0,1]$ the projection $p(\alpha):G\longrightarrow\eta(\alpha)$ is a group homomorphism.

\begin{proposition}
Let $G$ be a group, then
\begin{enumerate}[(1)]\sepa
\item
For every normal gradual subgroup $\sigma$ of $G$ we have $G/\sigma$ is a gradual quotient group of $G$.
\item
For every gradual quotient group $\eta$ of $G$ there is a normal gradual subgroup $\kappa$ of $G$, defined $\kappa(\alpha)=\Ker(G\longrightarrow\eta(\alpha))$, for any $\alpha\in(0,1]$.
\end{enumerate}
\end{proposition}

\begin{lemma}
Let $f:G\longrightarrow{G'}$ be a group homomorphism,
\begin{enumerate}[(1)]\sepa
\item
For any gradual subgroup $\sigma$ of $G$ we have that $f_*\sigma:(0,1]\longrightarrow{G'}$, defined $f_*(\sigma)(\alpha)=f(\sigma(\alpha))$ is a gradual subgroup of $G'$.
\item
For any gradual subgroup $\tau$ of $G'$ we have that $f^*\tau:(0,1]\longrightarrow{G}$, defined $f^*\tau(\alpha)=f^*(\tau(\alpha))$ is a gradual subgroup of $G$.
\item
If $\tau$ is normal, then $f^*\tau$ is normal.
\end{enumerate}
\end{lemma}

Let $\sigma_1,\sigma_2$ be gradual subgroups of $G$, we define $\sigma_1\subseteq\sigma_2$, and say $\sigma_1$ is a \textbf{subgroup} of $\sigma_2$, if $\sigma_1(\alpha)\subseteq\sigma_2(\alpha)$ for any $\alpha\in(0,1]$.

\begin{lemma}
Let $\sigma_1,\sigma_2$ be normal gradual subgroups of $G$, the following statements are equivalent:
\begin{enumerate}[(a)]\sepa
\item
$\sigma_1\subseteq\sigma_2$,
\item
For any $\alpha\in(0,1]$ there exist group homomorphisms $h_\alpha$ such that the following diagrams commute:
\[
\begin{xy}
\xymatrix{
1\ar[r]&\sigma_1(\alpha)\ar[rr]\ar@{^(->}[d]&&G\ar[rr]\ar@{=}[d]\ar[rr]&&G/\sigma_1(\alpha)\ar@{-->}[d]^{h_\alpha}\ar[r]&0\\
1\ar[r]&\sigma_2(\alpha)\ar[rr]&&G\ar[rr]&&G/\sigma_2(\alpha)\ar[r]&0
}\end{xy}
\]
\end{enumerate}
\end{lemma}

\begin{remark}
In consequence, to include $G/\sigma_2$ inside this theory we could introduce the notion of gradual quotient group of a gradual quotient group $G/\sigma_1$, hence study gradual objects which are not related to an ambient group. The same can be done in considering gradual subgroups. Thus, we could introduce the notion of gradual group or enriched gradual group, in a similar way as we did for gradual subsets and sets.
\end{remark}

If $\sigma_1$ and $\sigma_2$ are gradual subgroups of $G$, we have a gradual subset $\sigma_1\sigma_2$ of $G$, and not necessarily a gradual subgroup; we get a gradual subgroup whenever one of them is normal and, in this case, we have:

\begin{lemma}
Let $\sigma_1,\sigma_2$ be gradual subgroups of $G$ such that $\sigma_1$ is normal, then
\begin{enumerate}[(1)]\sepa
\item
$\sigma_1\sigma_2$ is a gradual subgroup of $G$.
\item
$\frac{\sigma_1\sigma_2}{\sigma_1}$, defined $\frac{\sigma_1\sigma_2}{\sigma_1}(\alpha)=\frac{\sigma_1(\alpha)\sigma_2(\alpha)}{\sigma_1(\alpha)}$ is a gradual subgroup of $\frac{G}{\sigma_1}$.
\end{enumerate}
\end{lemma}

This theory can be enriched whenever we consider maps between the different $\sigma(\alpha)$'s, i.e., enriched gradual subgroups. For instance, when there is an inclusion $\sigma(\beta)\subseteq\sigma(\alpha)$ whenever $\alpha\leq\beta$.

\subsection{Decreasing gradual subgroups}

A gradual subgroup $\sigma$ of $G$ is \textbf{decreasing} if for any $\alpha,\beta\in(0,1]$ such that $\alpha\leq\beta$, we have $\sigma(\beta)\subseteq\sigma(\alpha)$. For any decreasing gradual subgroup $\sigma$ of $G$,  for every $\alpha\in(0,1]$, we have that $\sigma(1)\subseteq\sigma(\alpha)$.

Let $\sigma$ be a gradual subgroup, we define the \textbf{accumulation} $\sigma^c$ of $\sigma$ as
\[
\sigma^c(\alpha)=\langle\cup\{\sigma(\beta)\mid\;\beta\geq\alpha\}\rangle,\textrm{ for any }\alpha\in(0,1].
\]
It is clear that $\sigma^c$ is a decreasing gradual subgroup, and a gradual subgroup $\sigma$ is decreasing if, and only if, $\sigma=\sigma^c$. In particular, we have the following properties of the operator $\sigma\mapsto\sigma^c$.

\begin{lemma}\label{le:20180421}
Let $G$ be a group, for every gradual subgroups $\sigma,\sigma_1,\sigma_2$ the following statements hold:
\begin{enumerate}[(1)]\sepa
\item
$\sigma\subseteq\sigma^c$.
\item
$\sigma^c=\sigma^{cc}$.
\item
If $\sigma_1\subseteq\sigma_2$, then $\sigma_1^c\subseteq\sigma_2^c$.
\item
$\sigma^c$ is the smallest decreasing gradual subgroup containing $\sigma$.
\item
If $\sigma$ is a normal gradual subgroup, then $\sigma^c$ is a normal subgroup, and $\sigma^c(\alpha)$ is the set of all products of elements in $\cup\{\sigma(\beta)\mid\;\beta\geq\alpha\}$ for any $\alpha\in(0,1]$.
\end{enumerate}
\end{lemma}
\begin{proof}
(5). %
Each element of $\langle\{\cup\sigma(\beta)\mid\;\beta\geq\alpha\}\rangle$ is a product $a_1*\cdots*a_t$, for some $a_i\in\sigma(\beta_i)$, and $\beta_i\geq\alpha$. For any $b\in{G}$ we have $g*a_1*\cdots*a_t*g^{-1}=(g*a_1*g^{-1})*\cdots*(g*a_t*g^{-1})\in\langle\{\cup\sigma(\beta)\mid\;\beta\geq\alpha\}\rangle$.
\end{proof}

This means that the map $\sigma\mapsto\sigma^c$ is a closure operator in the set $\mathcal{G}$ of all gradual subgroups of $G$ which is compatible with the product in $G$. The set of all $c$--closed gradual subgroups of $G$ is denoted by $\mathcal{J}(G)$.

\begin{proposition}\label{pr:20180421}
Let $\sigma_1,\sigma_2$ be gradual subgroups of $G$, then $\langle\sigma_1\sigma_2\rangle^c=\langle\sigma_1^c\sigma_2^c\rangle$. In addition, if either $\sigma_1$ or $\sigma_2$ is normal, then $(\sigma_1\sigma_2)^c=\sigma_1^c\sigma_2^c$.
\end{proposition}
\begin{proof}
In fact, we have that both, $\langle\sigma_1\sigma_2\rangle^c(\alpha)$ and $\langle\sigma_1^c\sigma_2^c\rangle(\alpha)$, are the subgroup generated by the subset $\cup\{\sigma_1(\beta)\cup\sigma_2(\beta)\mid\;\beta\geq\alpha\}$.
\par
Since $\sigma_1$ is normal each element of $(\sigma_1*\sigma_2)^c(\alpha)$ is a product $a*b$, for $a\in\sigma_1^c(\beta)$ and $b\in\sigma_2^c(\gamma)$, form some $\beta,\gamma\geq\alpha$, then $a*b\in\sigma_1^c(\alpha)*\sigma_2^c(\gamma)$. The converse is similar.
\end{proof}

In the same way, we may define the \textbf{strict accumulation} $\sigma^d$ of $\sigma$ as
\[
\sigma^d(\alpha)=\left\{\begin{array}{ll}
\sigma(1),&\textrm{ if }\alpha=1,\\
\langle\cup\{\sigma(\beta)\mid\;\beta>\alpha\}\rangle,&\textrm{ for any }\alpha\in(0,1).
\end{array}\right.
\]
We have that $\sigma^d$ is a decreasing gradual subgroup and $\sigma^d$ is normal whenever $\sigma$ is. Some properties of the operator $\sigma\mapsto\sigma^d$ are the following, whose proof is similar to the proof of Lemma~\eqref{le:20180421}, and Proposition~\eqref{pr:20180421}.

\begin{lemma}\label{le:20180423b}
Let $G$ be a group, for any gradual subgroups $\sigma,\sigma_1,\sigma_2$ the following statements hold:
\begin{enumerate}[(1)]\sepa
\item
$\sigma^d\subseteq\sigma^c$.
\item
If $\sigma_1\subseteq\sigma_2$, then $\sigma_1^d\subseteq\sigma_2^d$.
\item
$\sigma^d=\sigma^{dd}=\sigma^{cd}=(\sigma^c)^d$.
\item
If $\sigma$ is a normal gradual subgroup, then $\sigma^d$ is a normal subgroup, and $\sigma^c(\alpha)$  is the set of all products of elements in $\cup\{\sigma(\beta)\mid\;\beta>\alpha\}$ for any $\alpha\in(0,1]$.
\item
$(\sigma_1\sigma_2)^d=\langle\sigma_1^d\sigma_2^d\rangle$. In addition, if either $\sigma_1$ or $\sigma_2$ is normal, then $(\sigma_1\sigma)^d=\sigma_1^d\sigma_2^d$.
\end{enumerate}
\end{lemma}

A gradual subgroup $\sigma$ is an \textbf{strict decreasing gradual subgroup} whenever $\sigma=\sigma^{d}$,  and we have:

\begin{lemma}
For any gradual subgroup $\sigma$ we have that $\sigma^d$ is the biggest strict decreasing gradual subgroup contained in $\sigma^c$.
\end{lemma}

These results mean that the map $\sigma\mapsto\sigma^d$ is an interior operator in the set $\mathcal{J}^d(X)$ of all decreasing gradual subgroups of $G$.

\subsection{{Gradual subgroups and} fuzzy subgroups}

We shall show that there exists a strong relationship between gradual subgroups of a group $G$ and fuzzy subgroups of $G$. Remember that if $G$ is a group, a \textbf{fuzzy subgroup} $\mu$ of $G$ is a nonconstant, equal to 0, map $\mu:G\longrightarrow[0,1]$ satisfying $\mu(xy^{-1})\geq\mu(x)\wedge\mu(y)$, for any $x,y\in{G}$. In particular, if $e$ is the neutral element of $G$, then $\mu(e)\geq\mu(x)$ and $\mu(x)=\mu(x^{-1})$ for any $x\in{G}$. Our aim is to identify fuzzy subgroups with some particular decreasing gradual subgroups.

First, we need to realize some modifications to have well defined gradual groups starting from a fuzzy group.

In the set of all fuzzy subgroups $\mu$ of $G$, we define a equivalence relation: $\mu_1\sim\mu_2$ if $\mu_1(x)=\mu_2(x)$ for any $x\neq{e}$. In order to choose a canonical element in each equivalence class, following an idea in \cite{Jahan:2012}, for any fuzzy subgroup $\mu$ we define $\mu^1$ as follows: $\mu^1(x)=\left\{\begin{array}{ll}
\mu(x)&\textrm{ if }x\neq{e},\\
1&\textrm{ if }x=e.
\end{array}\right.$
Observe that each equivalence class $[\mu]$ has a unique element of the shape $\mu^1$, whenever $\mu^1$ is a fuzzy subgroup.

\begin{lemma}
Let $\mu$ be a fuzzy subgroup of a group $G$, then $\mu^1$ is a fuzzy subgroup, and $\mu^1$ is normal whenever $\mu$ is.
\end{lemma}
\begin{proof}
Let $x,y\in{G}$, then $\mu^1(xy^{-1})=\mu(xy^{-1})\geq\mu(x)\wedge\mu(y)=\mu^1(x)\wedge\mu^1(y)$ whenever $x,y,xy^{-1}\neq{e}$.
On the other hand, if $xy^{-1}=e$, then $\mu^1(xy^{-1})=1\geq\mu^1(x)\wedge\mu^1(y)$; if $x\neq{e}$, $y=e$, then $\mu^1(y)=1\geq\mu^1(x)$, and we have $\mu^1(xy^{-1})=\mu^1(x)=\mu^1(x)\wedge\mu^1(y)$.
\end{proof}

A decreasing gradual subgroup $\sigma$ satisfies \textbf{property (F)} if there exists $\Max\{\alpha\mid\;x\in\sigma(\alpha)\}$ for every $x\in\cup_\alpha\sigma(\alpha)$. An strict decreasing gradual subgroup $\sigma$ satisfies \textbf{property (inf--F)} if $\gamma=\Inf\{\beta\mid\;x\notin\sigma^d(\beta)\}$ satisfies $x\in\sigma(\gamma)$ for any $x\in{G}$.

Let $\sigma$ be a decreasing gradual subgroup, let as denote $\sigma^*(\alpha)=\sigma^c(\alpha)\setminus\sigma^d(\alpha)$, the difference set, for every $\alpha\in(0,1]$.

\begin{lemma}\label{le:20180423}
Let $\sigma$ be a decreasing gradual subgroup such that $\sigma^c\neq\sigma^d$, the following statements are equivalent:
\begin{enumerate}[(a)]\sepa
\item
$\sigma^c$ satisfies property (F).
\item
$\sigma^d$ satisfies property (inf-F).
\item
$\stackrel{\bullet}{\cup}_{\alpha\in(0,1]}\sigma^*(\alpha)=\cup\{\sigma(\alpha)\mid\;\alpha\in(0,1]\}$.
\end{enumerate}
\end{lemma}

Let $[\mu]$ be the equivalence class of the fuzzy subgroup $\mu$; we define a gradual subset $\sigma(\mu)$ of $G$ by:
\[
\sigma(\mu)(\alpha)=\{x\in{G}\mid\;\mu^1(x)\geq\alpha\}.
\]

\begin{lemma}
The map $\nu:[\mu]\mapsto\sigma(\mu)$ is well defined, and $\sigma(\mu)$ is a decreasing gradual subgroup satisfying property (F).
\end{lemma}
\begin{proof}
Observe that for any class $[\mu]$ there exists only one fuzzy subgroup $\mu$ such that $\mu=\mu^1$.
\end{proof}

Let $\sigma$ be a decreasing gradual subgroup satisfying property (F); we define a fuzzy subset $\mu(\sigma)$ by,
\[
\mu(\sigma)(x)=\left\{\begin{array}{ll}
\Max\{\gamma\mid\;x\in\sigma(\gamma)\},&\textrm{ if }x\neq{e},\\
1,&\textrm{ if }x=e.
\end{array}\right.
\]

\begin{lemma}
With the above notation $\mu(\sigma)$ is a fuzzy subgroup, and we have a map $\upsilon:\sigma\mapsto[\mu(\sigma)]$ from the set of all decreasing gradual subsets satisfying property (F) to the set of all classes of fuzzy subgroups.
\end{lemma}

Now we have the announced relationship of gradual subgroups and fuzzy subgroups.

\begin{theorem}\label{th:201804b}
Let $G$ be a group, the maps $\nu:[\mu]\mapsto\sigma(\mu)$ and $\upsilon:\sigma\mapsto[\mu(\sigma)]$ defines a bijective correspondence between:
\begin{enumerate}[(1)]\sepa
\item
equivalence classes of fuzzy subgroups $[\mu]$ of $G$ and
\item
descending gradual subgroups $\sigma$ of $G$ satisfying property (F).
\end{enumerate}
\end{theorem}
\begin{proof}
Let $\mu=\mu^1$ be a fuzzy group, for any $x\in{G}$ we have:
\[
\mu(\sigma(\mu))(x)
=\Max\{\alpha\mid\;x\in\sigma(\mu)(\alpha)\}
=\Max\{\alpha\mid\;\alpha\leq\mu(x)\}
=\mu(x).
\]
On the other hand, let $\sigma$ be a decreasing gradual subgroup satisfying property (F), for any $\alpha\in(0,1]$ we have:
\[
\sigma(\mu(\sigma))(\alpha)
=\{x\mid\;\mu(\sigma)(x)\geq\alpha\}
=\{x\mid\;\Max\{\beta\mid\;x\in\sigma(\beta)\}\geq\alpha\}
=\sigma(\alpha).
\]
\end{proof}

\begin{lemma}
Let $\mu_1,\mu_2$ fuzzy subgroups, let us define $[\mu_1]\,[\mu_2]$ as $[\mu_1\,\mu_2]$.
\end{lemma}
\begin{proof}
The product $[\mu_1]\,[\mu_2]$ is well defined. Let $[\mu]=[\mu']$, for any $\mu_2$ se have:
\[
\mu\,\mu_2(x)
=\Sup\{\mu(y)\wedge\mu_2(z)\mid\;yz=x\}
=\Sup\{\mu'(y)\wedge\mu_2(z)\mid\;yz=x\}
=\mu'\,\mu_2(x).
\]
\end{proof}

\begin{remark}
Unfortunately, in Theorem~\eqref{th:201804b} the  map $\nu$ is not a homomorphism with respect to the product of classes of fuzzy subgroups. Indeed, for any $\mu_1,\mu_2$ we have:
\[
\begin{array}{ll}
(\sigma(\mu_1)\,\sigma(\mu_2))(\alpha)
&=\sigma(\mu_1)(x)\,\sigma(\mu_2)(\alpha)
=\{x\mid\;\mu_1(x)\geq\alpha\}\,\{x\mid\;\mu_2(x)\geq\alpha\}\\
&\subseteq\{x\mid\;\Sup\{\mu_1(y)\wedge\mu_2(z)\mid\;y\,z=x\}\geq\alpha\}
=\{x\mid\;(\mu_1\,\mu_2)(x)\geq\alpha\}\\
&=\sigma(\mu_1\,\mu_2)(x).
\end{array}
\]
\end{remark}

This inclusion could be strict as the following example shows.

\begin{example}
We define fuzzy subgroups $\mu_1$ and $\mu_2$ of $\mathbb{Z}$ as follows:
\[
\mu_1(x)=\left\{\begin{array}{ll}
0,&\textrm{ if }x\in\mathbb{Z}\setminus2\mathbb{Z},\\
1-\frac{2^t}{3^t},&\textrm{ if }x\in2^t\mathbb{Z}\setminus2^{t+1}\mathbb{Z},\\
1,&\textrm{ if }x=0.
\end{array}\right.
\quad
\mu_2(x)=\left\{\begin{array}{ll}
0,&\textrm{ if }x\in\mathbb{Z}\setminus3\mathbb{Z},\\
\frac{1}{2}-\frac{1}{3^t},&\textrm{ if }x\in3^t\mathbb{Z}\setminus3^{t+1}\mathbb{Z},\\
1,&\textrm{ if }x=0.
\end{array}\right.
\]
We claim $(\mu_1+\mu_2)(2)=\Sup\{\mu_1(y)\wedge\mu_2(2-y)\mid\;y\in\mathbb{Z}\}\leq\frac{1}{2}$. Indeed, we have two possibilities:
\begin{enumerate}[(1)]\sepa
\item
$\mu_1(y)>\frac{1}{2}$, then $y\in4\mathbb{Z}$, i.e., there exists $k\in\mathbb{Z}$ such that $y=4k$. Hence, $\mu_2(2-y)=\mu_2(2-4k)=\mu_2(2(1-2k))<\frac{1}{2}$ as $2-y\neq0$.
\item
$\mu_1(y)<\frac{1}{2}$.
\end{enumerate}
In both cases we have $\mu_1(y)\wedge\mu_2(2-y)<\frac{1}{2}$. In addition, we can choose $y$ such that $\mu_1(y)\wedge\mu_2(2-y)$ is as closed to $\frac{1}{2}$ as we desire. For any $2\leq{t,s}\in\mathbb{N}$ there exist $k,h\in\mathbb{Z}$ such that $2^{t-1}k-3^sh=1$, hence $2-2^tk=2(1-2^{t-1}k)=3^sh$; now, if we take $y=2^tk$, then $\mu_1(y)\geq1-\frac{1}{3^t}$ and $\mu_2(2-y)\geq\frac{1}{2}-\frac{1}{3^t}$. In consequence, $\frac{1}{2}>\mu_1(y)\wedge\mu_2(2-y)\geq\frac{1}{2}-\frac{1}{3^t}$, which implies that $(\mu_1+\mu_2)(2)=\frac{1}{2}$, and $2\in(\mu_1+\mu_2)_{\frac{1}{2}}$. On the other hand, we have $(\mu_1)_{\frac{1}{2}}+(\mu_2)_{\frac{1}{2}}=4\mathbb{Z}$, and $2\notin(\mu_1)_{\frac{1}{2}}+(\mu_2)_{\frac{1}{2}}$.
\end{example}

We shall change the assignation defined by $\nu$ to consider $\widetilde{\nu}:[\mu]\mapsto\widetilde{\sigma}(\mu)$, in which $\widetilde{\sigma}(\mu)$ is a strict decreasing gradual subgroup satisfying property (inf-F), and is defined by:
\[
\widetilde{\sigma}(\mu)(\alpha)=\left\{\begin{array}{ll}
\{x\in{G}\mid\;\mu^1(x)>\alpha\},&\textrm{ if }\alpha\neq1,\\
\{x\in{G}\mid\;\mu^1(x)=1\},&\textrm{ if }\alpha=1,
\end{array}\right.
\]
whose inverse is $\widetilde{\upsilon}:\sigma\mapsto[\widetilde{\mu}(\sigma)]$, defined
\[
\widetilde{\mu}(\sigma)(x)=\left\{\begin{array}{ll}
\Inf\{\gamma\mid\;x\notin\sigma(\gamma)\},&\textrm{ if }x\neq{e},\\
1,&\textrm{ if }x=e.
\end{array}\right.
\]
Thus we have the following theorem

\begin{theorem}\label{th:201804}
With the above notation we have:
\begin{enumerate}[(1)]
\item
$\widetilde{\sigma}(\mu)$ is a strict decreasing gradual subgroup satisfying property (inf-F), and $\widetilde{\nu}$ is well defined.
\item
$\widetilde{\mu}(\sigma)$ is a fuzzy subgroup.
\item
the maps $\widetilde{\nu}$ and $\widetilde{\upsilon}$ define a bijective correspondence between equivalence classes of fuzzy subgroups $[\mu]$ of $G$ and strict descending gradual subgroups $\sigma$ of $G$ satisfying property (inf-F).
\item
$\widetilde{\nu}$ is a homomorphism with respect to the product of classes of fuzzy subgroups.
\end{enumerate}
\end{theorem}
\begin{proof}
(1). %
First we observe that $\widetilde{\sigma}(\mu)=\sigma(\mu)^d$, hence it is a strict gradual subgroup, and by Lemma~\eqref{le:20180423} it satisfies property (inf-F). It is well defined as $\sigma(\mu)$ is uniquely defined, hence it is $\widetilde{\sigma}(\mu)$.
\par (2). %
It is a direct consequence of Lemma~\eqref{le:20180423b}.
\par (3). %
We can mimic the proof of Theorem~\eqref{th:201804b}.
\par (4). %
For any $\mu_1,\mu_2$ we have:
\[
\begin{array}{ll}
(\sigma(\mu_1)\,\sigma(\mu_2))(\alpha)
&=\sigma(\mu_1)(x)\,\sigma(\mu_2)(\alpha)
=\{x\mid\;\mu_1(x)>\alpha\}\,\{x\mid\;\mu_2(x)>\alpha\}\\
&=\{x\mid\;\Sup\{\mu_1(y)\wedge\mu_2(z)\mid\;y\,z=x\}>\alpha\}
=\{x\mid\;(\mu_1\,\mu_2)(x)>\alpha\}\\
&=\sigma(\mu_1\,\mu_2)(x).
\end{array}
\]
\end{proof}

\subsection{Normal {fuzzy} subgroups}

A fuzzy subgroup $\mu$ of a group $G$ is \textbf{normal} if $\mu\,\eta=\eta\,\mu$ for any fuzzy subset $\eta$, or equivalently if $\mu(x\,y)=\mu(y\,x)$ for any $x,y\in{G}$, see \cite{MORDESON/MALIK:1998}. We are interesting in relating normal fuzzy subgroups and normal gradual subgroups. We have defined a gradual subgroup $\sigma$ to be \textbf{normal} if $\sigma(\alpha)\subseteq{G}$ is a normal subgroup for any $\alpha\in(0,1]$.

\begin{lemma}
Let $\mu_1,\mu_2$ be a fuzzy subgroup such that $\mu_1\sim\mu_2$ and $\mu_1$ is normal, then $\mu_2$ is normal.
\end{lemma}
\begin{proof}
By hypothesis $\mu_1(x\,y)=\mu_1(y\,x)$ for every $x,y\in{G}$, if $x\,y,y\,x\neq{e}$, then $\mu_2(x\,y)=\mu_2(y\,x)$. If $x\,y=e$, then $x=y^{-1}$, hence $y\,x=e$, and we have $\mu_2(x\,y)=\mu_2(y\,x)$
\end{proof}

As a consequence, if $\mu$ is a normal fuzzy subgroup, then every fuzzy subgroup in $[\mu]$ is normal; in particular $\mu^1$ is normal.

\begin{theorem}
Let $\mu$ be a fuzzy subgroup, the following statements are equivalent:
\begin{enumerate}[(a)]\sepa
\item
$\mu$ is normal.
\item
$\widetilde{\sigma}(\mu)$ is normal.
\end{enumerate}
\end{theorem}
\begin{proof}
We may assume, without loss of generality, that $\mu=\mu^1$. Let $g\in{G}$ and let us consider the fuzzy subset $\eta(g)$ defined as the characteristic function of $\{g\}$, then
$$
(\eta(g)\,\mu)(x)
=\Sup\{\eta(g)(x_1)\wedge\mu(x_2)\mid\;x=x_1\,x_2\}
=\eta(g)(g)\wedge\mu(g^{-1}\,x)=\mu(g^{-1}\,x),
$$
and in the same way $(\mu\,\eta(g))(x)=\mu(x\,g^{-1})$. Then
$$
\widetilde{\sigma}(\mu)
=\widetilde{\sigma}(\eta(g^{-1})\,\mu\,\eta(g))
=\widetilde{\sigma}(\eta(g^{-1}))\,\widetilde{\sigma}(\mu)\,\widetilde{\sigma}(\eta(g))
=g^{-1}\,\widetilde{\sigma}(\mu)\,g.
$$
Therefore, $\widetilde{\sigma}(\mu)$ is normal. Conversely, if $\widetilde{\sigma}(\mu)$ is normal, for any element $g\in{G}$ we have $g^{-1}\,\widetilde{\sigma}(\mu)\,g=\widetilde{\sigma}(\mu)$, hence $\eta(g^{-1})\,\mu\,\eta(g)=\mu$, and $\mu$ is a normal fuzzy subgroup.
\end{proof}

\subsection{Gradual groups}

From any decreasing gradual subgroup $\sigma$  of a group $G$ we have two families: one $\{\sigma(\alpha)\mid\;\alpha\in(0,1]\}$ is a family of groups and the other $\{i_{\alpha,\beta}:\sigma(\beta)\hookrightarrow\sigma(\alpha)\mid\;\alpha\leq\beta\}$ is the family of the inclusions. To include these objects inside a more general theory, we shall consider contravariant functors from $(0,1]$ to $\Gr$, the category of groups.

For any contravariant functor $F:(0,1]\longrightarrow\Gr$ and every $\alpha\leq\beta$ we have now a group homomorphism from $F(\beta)$ to $F(\alpha)$, and the pair $(\{F(\alpha)\mid\;\alpha\in(0,1]\},\{F(f_{\alpha,\beta})\mid\;\alpha\leq\beta\})$ is a direct systems of groups and group homomorphisms, hence there exists its direct limit, say $\dlim{F}$.

We define a \textbf{gradual group} as a contravariant functor $F:(0,1]\longrightarrow\Gr$, and a gradual group homomorphisms from $F_1$ to $F_2$ is just a natural transformation from $F_1$ to $F_2$. Therefore, we can consider the category of gradual groups and gradual group homomorphisms, which we denote by $\mathcal{G}$.

An example of such a gradual group is provided by any decreasing gradual subgroup $\sigma$ of a group $G$. In this case, the direct limit $\dlim\sigma$ is isomorphic to a subgroup of $G$; indeed, it is the union $\cup\{\sigma(\alpha)\mid\;\alpha\in(0,1]\}$.

Following this example, for any arbitrary gradual group $F$, we say $F$ is a \textbf{decreasing gradual group} whenever each $F(f_{\alpha,\beta})$, for $\alpha\leq\beta$. The class of all decreasing gradual groups is denoted by $\mathcal{J}$. To well understand the structure of decreasing gradual groups, we build an operator (an endofunctor) $d$ in $\mathcal{J}$; defined on objects as follows: for any $F\in\mathcal{J}$ we define $F^d(\alpha)=\dlim_{(\alpha,1]}F(\gamma)$, for every $\alpha\in(0,1]$. We collect these results in the following proposition, whose proof, after the theory developed in section~\eqref{se:20181115}, is straightforward.

{\begin{proposition}
Let $F$ be a decreasing gradual group, and $\theta:F_1\longrightarrow{F_2}$ be a decreasing gradual map. The following statements hold.
\begin{enumerate}[(1)]\sepa
\item
$F^d$ is a decreasing gradual group.
\item
$d$ is an endfunctor of the full subcategory $\mathcal{J}$ of $\mathcal{G}$.
\item
$d$ is an interior operator in $\mathcal{J}$.
\end{enumerate}
\end{proposition}}

A \textbf{strict decreasing gradual group} is a decreasing gradual group $F$ such that $F=F^d$.

At this point it is convenient to remark that we have gradual groups and gradual subgroups. Contrary to decreasing gradual subgroups, that need of an ambient or a ground group, decreasing gradual groups have it included: it is the direct limit of the direct system that the gradual group defines. This situation allows us to formulate a more attractive category theory of gradual objects which includes the usual constructions of the category of groups. In this context, decreasing gradual groups, strict decreasing gradual groups and fuzzy groups can be identified with adequate subcategories, see the forthcoming paper \cite{Garcia/Jara:2018b}.

\subsection*{Conclusion}

Our aim in this paper has been to introduce more general notions that fuzzy subset in order to find a framework in which develop an easier theory that allows new techniques to prove and establish new results in fuzzy theory. In this sense we start from the concept of gradual element with the goal of introducing gradual subsets. At this point we establish a bijective correspondence between fuzzy subsets and a particular kind of gradual subsets (strictly decreasing gradual subsets), that satisfies property (inf--F). The more interesting property of this correspondence is that it preserves arbitrary unions and intersections of fuzzy subsets.

In a second degree of abstraction we consider a gradual subset as a contravariant functor from the category $(0,1]$ to the category of sets, which allows us to define the notion of fuzzy and gradual sets without the use of an ambient set. Thus we have three degrees of abstraction, the first one corresponds to fuzzy subset; the second one to gradual subsets, identifying fuzzy subsets as some particular gradual subsets; and the third one to contravariant functors from $(0,1]$ to the category of sets, or directed systems of sets, identifying decreasing gradual subsets as those systems with injective maps. Observe that in each abstraction level we have the objects studied in the previous one. We also establish the corresponding theory for groups in two different but compatible ways; (1) defining contravariant functors to the category of groups; $\mathcal{G}r^{(0,1]}$, and (2) defining groups in the functor category $\mathcal{S}et^{(0,1]}$.

One of the goals of this paper was to find a framework in which to study together the two crisp sets associated with each fuzzy set, and we have proven that groups and gradual groups allow it to do so. On the other hand, the use of category theory tools will allow to extend this working method to other structures, of which the sets and groups studied are only an example.

%%\providecommand{\bysame}{\leavevmode\hbox to3em{\hrulefill}\thinspace}
%%\providecommand{\MR}{\relax\ifhmode\unskip\space\fi MR }
%%% \MRhref is called by the amsart/book/proc definition of \MR.
%%\providecommand{\MRhref}[2]{%
%%  \href{http://www.ams.org/mathscinet-getitem?mr=#1}{#2}
%%}
%%\providecommand{\href}[2]{#2}

%\bibliography{libros,articulo}
%\bibliographystyle{amsplain}
\end{document}